\documentclass[a4paper,12pt]{article}
\usepackage{amsfonts,amsmath,amssymb,amsthm}
\usepackage{txfonts}
\usepackage[utf8]{inputenc}
\usepackage{amsmath}
\usepackage{amssymb}
\usepackage{mathtools}
\usepackage{bbm}

\RequirePackage[colorlinks,citecolor=red,urlcolor=blue, linkcolor=blue]{hyperref}
\usepackage[margin=0.8in]{geometry}

\usepackage[dvips]{graphics}
\usepackage{epsfig,rotating}
\usepackage{tabularx}
\usepackage{bm}
\usepackage{bbm}
\usepackage{array}
\newcolumntype{P}[1]{>{\raggedright\arraybackslash}p{#1}}

\usepackage{mdframed}
\usepackage{lipsum}
\usepackage{enumitem}

\usepackage[authoryear,square]{natbib}

\usepackage[nottoc,numbib]{tocbibind} 
\usepackage[toc,page]{appendix}

\usepackage[dvips]{graphics}
\usepackage{multirow}
\usepackage{array}
\usepackage{epsfig,rotating}
\usepackage{amssymb,amsmath}
\usepackage{latexsym}
\usepackage{bm}
\usepackage{enumitem}

\usepackage{txfonts}
\usepackage[utf8]{inputenc}

\newcommand{\mc}{\mathcal}

\numberwithin{equation}{section}

\newcommand{\la}{\left\{}
\newcommand{\ra}{\right\}}
\newcommand{\lb}{\left(}
\newcommand{\rb}{\right)}
\newcommand{\mb}{\mathbb}
\newcommand{\ve}{\varepsilon}

\newcommand{\argmax}{\operatornamewithlimits{argmax}}

\newtheorem{theorem}{Theorem}[section]	
\newtheorem{lemma}{Lemma}[section]

\newtheorem{proposition}{Proposition}[section]

\newtheorem{remark}{Remark}[section]

\renewcommand{\P}{\mathbb{P}}

\newcommand{\Var}{\operatorname{Var}}

\newcommand{\Frechet}{Fr\'{e}chet }
\newcommand{\cid}{\stackrel{\rm d}{\rightarrow}}

\begin{document}
\title{Maximum of sparsely equicorrelated Gaussian fields and applications}
\author{
Johannes Heiny \\ \small{Department of Mathematics, KTH Royal Institute of Technology, Stockholm} \\ \small{\href{mailto:heiny@kth.se}{heiny@kth.se}} 
\and Tiefeng Jiang \\ \small{School of Data Science, Chinese University of Hong Kong, Shenzhen} \\ \small{\href{mailto:jiang040@cuhk.edu.cn}{jiang040@cuhk.edu.cn}}
\and Tuan Pham\thanks{Corresponding author.} \\ \small{Department of Statistics Data Science, University of Texas, Austin} \\  \small{\href{mailto:tuan.pham@utexas.edu}{tuan.pham@utexas.edu}}
\and Yongcheng Qi \\ \small{Department of Mathematics and Statistics, University of Minnesota, Duluth} \\  \small{\href{mailto:yqi@d.umn.edu}{yqi@d.umn.edu} }
}


\maketitle

\begin{abstract}

We investigate the extreme values of a sparse and equicorrelated Gaussian field on a triangle: the correlations on every vertical or horizontal line are all equal to a parameter $r \in [0,1/2]$ and are zero everywhere else. This problem is closely linked with various problems in high-dimensional statistics and extreme-value theory.  We identify the threshold for $r$ at which the standard Gumbel law breaks down. Our result is based on a subtle application of the Chen-Stein method for Poisson approximation. As applications, we discuss the implication of our results on multiple testing and resolve several questions that were left open in \cite{heiny2024maximum}, \cite{tang2022asymptotic} and \cite{Jiang19}.

\end{abstract}





\tableofcontents
\section{Introduction} \label{Intro}

Consider a centered Gaussian field $\mathcal{G}_n = \la  G_{ij} \ra_{1 \leq i<j \leq n}$ with correlation structure given by
\begin{align} \label{sparse equi GF}
    \mb{E} \lb G_{ij} G_{kl} \rb = \begin{cases}
        0~~  &\text{if} \ \big|\la i,j \ra \cap \la k,l \ra \big|= 0; \\
        r~~  &\text{if} \ \big|\la i,j \ra \cap \la k,l \ra\big| =1; \\
        1~~ &\text{if} \ \big|\la i,j \ra \cap \la k,l \ra\big| =2.
    \end{cases} 
\end{align}
The field $\mathcal{G}_n$ can be regarded as a Gaussian field defined in the triangle such that all the elements belonging to the same row or column are equicorrelated with a common parameter $r \in [0,1/2]$, and are independent otherwise. The restriction that $r \in [0,1/2]$ is natural since the correlation structure of $\mathcal{G}_n$ is not well-defined if $r$ is negative or exceeds $1/2$. The requirement $r \leq 1/2$, for instance, can be seen by noting that
\[
\mbox{Var} \lb G_{12} -G_{23} + G_{34} - G_{41}  \rb =4 - 8r \geq 0. 
\]

\noindent Despite having an exotic form, the Gaussian field $\mathcal{G}_n$ in \eqref{sparse equi GF} appears in a number of different problems in high-dimensional statistics and extreme-value theory. In particular, understanding the extreme values of $\mathcal{G}_n$ leads to improvement of many existing results. To the best of our knowledge, all existing results that require dealing with the correlation structure \eqref{sparse equi GF} are limited to the case $r \leq 1/3$, where the maximum of $\mathcal{G}_n$ behaves like that of i.i.d. standard normals. Let us describe in what follows several contexts in which \eqref{sparse equi GF} shows up. 
\begin{itemize}
\item {\it \textbf{Maximum interpoint distance in high-dimensions.}} In \cite{heiny2024maximum, tang2022asymptotic}, the authors investigate the asymptotic distribution of the statistic
\begin{align} \label{Dn}
D_n := \max_{1 \leq i < j \leq p} \| \bm{X}_i - \bm{X}_j \|
\end{align}
where the vectors \( \{ \bm{X}_i = ( x_{k1},\dots,x_{kn})  \}_{i=1}^p \) are i.i.d.\ in \( \mathbb{R}^n\), and the coordinates are i.i.d.\ with mean zero and unit variance.  { We note that, in \eqref{Dn}, $n$ denotes the dimension of the data, while $p$ denotes the sample size. 
We adopt this notation from \cite{heiny:kleemann:2024}. 
By contrast, \cite{tang2022asymptotic} uses the opposite convention, where $n$ is the sample size and $p$ is the dimension. Despite the difference in the notation, the condition \eqref{4th moment <5} below is the same in both works.}

It was shown in \cite{heiny2024maximum} and \cite{tang2022asymptotic} that $D_n$ asymptotically follows a Gumbel distribution under the high-dimensional settings, appropriate moment conditions  and 
\begin{align} \label{4th moment <5}
\mb{E} x_{11}^4 \leq 5.
\end{align}
The asymptotic distribution of \( D_n \) when~\eqref{4th moment <5} fails was left open in these works. Notably, the analysis of \( D_n \) in these papers essentially reduces to studying the maximum of the Gaussian field~\( \mathcal{G}_n \) in~\eqref{sparse equi GF}, with correlation parameter
\[
r = \frac{\mathbb{E}[x_{11}^4] - 1}{2\mathbb{E}[x_{11}^4] + 2}.
\]
The condition~\eqref{4th moment <5} is equivalent to requiring \( r \leq \tfrac{1}{3} \).

\item {\it \textbf{Sample coefficients of equicorrelated populations.}} In \cite{Jiang19}, the authors investigated the asymptotic distribution of the largest sample coefficients in the ultra high-dimensional settings. This includes the largest entries of sample covariance and correlation matrices described in \eqref{sample covariance} and \eqref{pearson} below. Some background and related results are presented in Section \ref{sec-sample coeff}.

Let us briefly introduce the settings. Given $n$ i.i.d. data points $\bm{X}_k =\lb x_{k1},x_{k2},\dots,x_{kp} \rb^\top$, $1 \leq k \leq n$ from a multivariate normal distribution with covariance matrix
\[
\mb{E} \lb \bm{X}_1\bm{X}_1^\top \rb = \lb 1 - \rho \rb \bm{I}_p + \rho \bm{1}\bm{1}^\top\,,
\]
define the sample correlation and (non-normalized) sample covariance coefficients
 \begin{align}
      \hat{\rho}_{ij} &:= \frac{ \sum_{k=1}^{n} ( x_{ki} - \bar{x}_i ) ( x_{kj} - \bar{x}_j )}{\sqrt{\sum_{k=1}^{n} ( x_{ki} - \bar{x}_i )^2} \cdot \sqrt{\sum_{k=1}^{n} ( x_{kj} - \bar{x}_j )^2}} \,,  \label{pearson} \\
       \hat{r}_{ij} &:= \sum_{k=1}^{n} ( x_{ki} - \bar{x}_i ) ( x_{kj} - \bar{x}_j ). \label{sample covariance}
  \end{align}
 It was shown in~\cite{Jiang19} that 
\(\max_{i<j} \hat{\rho}_{ij}\) and \(\max_{i<j} \hat{r}_{ij}\) 
are asymptotically Gumbel, normal, or a mixture of the two in the high-dimensional setting, depending on whether 
\(\rho \sqrt{\log p}\ll 1\), \(\rho\sqrt{\log p} \gg 1\), or \(\rho \sqrt{\log p} \to c \in (0,\infty)\), respectively. 

Interestingly, the analysis of both 
\(\max_{i<j} \hat{\rho}_{ij}\) and \(\max_{i<j} \hat{r}_{ij}\) 
can be reduced to the study of the maximum of the Gaussian field \(\mathcal{G}_n\) in~\eqref{sparse equi GF}, 
with correlation parameter
\[
r = \frac{\rho}{1+\rho} \quad \text{for } \eqref{sample covariance}, 
\qquad 
r = \frac{\rho - \rho^2/2}{1 + 2\rho - \rho^2/2} \quad \text{for } \eqref{pearson}.
\]
Note that, as \(\rho \to 1\), the parameter \(r\) approaches \(1/2\) in the case of \eqref{sample covariance}, 
and \(1/5\) in the case of \eqref{pearson}.
Therefore, the technical condition
\begin{align} \label{<1/2}
\limsup \rho < \tfrac{1}{2}
\end{align}
was imposed in~\cite{Jiang19} for their analysis of \(\max_{i<j} \hat{r}_{ij}\). 
It is worth pointing out that this assumption is equivalent to requiring \( r \leq \tfrac{1}{3} \) in~\eqref{sparse equi GF}. 

Interestingly, it can be seen that the corresponding $r$ for \eqref{pearson} is always less or equal to $1/5$ in the Gaussian case, which explains why no restriction on $\rho$ is needed for the analysis of $\max_{i<j} \hat{\rho}_{ij}$ in \cite{Jiang19}.

\item {\it \textbf{Family-wise error rate (FWER) control in mutiple testing.}} Consider the multiple testing problem
\[
H_{0ij}: \theta_{ij} = 0  \ \text{against} \ H_{1ij}: \theta_{ij}>0
\]
where one obtains a single observation $Y_{ij} \sim N \lb \theta_{ij},1 \rb$ for each $1\leq i<j \leq n$. 


Suppose the joint distribution of $\mc{Y}_{n} = \la Y_{ij}; 1\leq i<j \leq n \ra$ is a Gaussian graphical model with covariance satisfying
\begin{align} \label{graphical covariance}
    \mb{E} Y_{ij} Y_{kl} := \begin{cases}
        0 \ &\text{if} \ |\la i,j \ra \cap \la k,l \ra| = 0; \\
        r_{ijkl} \in [0,1/2-\delta]  &\text{if} \  |\la i,j \ra \cap \la k,l \ra| = 1;\\
        1  &\text{if} \  |\la i,j \ra \cap \la k,l \ra| = 2
    \end{cases}
\end{align}
for some $\delta>0$.

The graphical covariance structure described above commonly arises in brain imaging data: adjacent brain regions are typically correlated, while non-adjacent regions tend to be independent. A natural test statistic for controlling the family-wise error rate (FWER) is to reject if
\begin{align} \label{max test}
\max_{1\leq i<j\leq n} Y_{ij} > u_n
\end{align}
for some threshold $u_n$.
It is therefore of interest to understand the behavior of $\max_{1\leq i<j\leq n} Y_{ij}$ under this graphical model. Standard methods based on excursion probabilities of smooth Gaussian fields \cite{cheng2016excursion,cheng2019multiple,cheng2024expected} rely on smooth manifold structures and are not applicable here due to the discrete nature of the graph. Nevertheless, our main result shows that analyzing the maximum of $\mathcal{G}_n$ with correlation structure \eqref{sparse equi GF} yields an asymptotically exact approximation for $u_n$.

In general, controlling the FWER under complex dependence is challenging: one often resorts to union bounds, which give overly conservative thresholds. For this reason, much of the multiple testing literature instead focuses on controlling the false discovery rate (FDR). As an application of our results, we demonstrate that asymptotically exact thresholds for FWER control can in fact be obtained for the Gaussian graphical model \eqref{graphical covariance}.


  
\end{itemize}
The goal of the present article is to resolve all these gaps by establishing new results on the maximum of the Gaussian field~$\mathcal{G}_n$ in~\eqref{sparse equi GF}. We show that, perhaps surprisingly, the maximum of~$\mathcal{G}_n$ still behaves like that of i.i.d.\ standard normal random variables as long as 
\[
1-2r \gg \frac{\log \log n}{\sqrt{\log n}}.
\]
This constrasts the common belief that the Gumbel asymptotic distribution no longer holds when $r>1/3$. 

Moreover, at the finer scale where $(1-2r)\log n \to \lambda \in [0,\infty)$, the i.i.d. behavior breaks down. In this regime, we show that the asymptotic distribution of the maximum resembles either the right-most point of a perturbed Poisson process or the sum of the two right-most points of a Poisson process. Our proof relies on a delicate application of Chen–Stein’s method for Poisson approximation via a carefully designed truncation argument to create asymptotic independence.

 As a consequence of the newly developed results, we show the following.
\begin{enumerate}
    \item The assumption for the maximum interpoint distance $D_n$ that the fourth moment is bounded above by $5$, that is \eqref{4th moment <5}, can be removed. Thus, the results in \cite{heiny2024maximum,tang2022asymptotic} hold without this assumption. 
    \item The i.i.d. behavior breaks down and there are different limiting distributions for $D_n$ when the fourth moment diverges to infinity at the scale $\log p$.
    \item For sample coefficients, we recover  the results in \cite{Jiang19} for non-Gaussian distributions without imposing the restriction that the correlation parameter $\rho$ is bounded away from $1/2$ in \eqref{<1/2}. Morever, we show that in the non-Gaussian setting, the situation is far more complicated: the phase transition involving a non-Gumbel distribution and central limit theorem can still hold even in the weakly dependent regime $\rho \sqrt{\log p} \to 0$ if the fourth moment of the marginal is allowed to diverge.

   In contrast, the limiting distribution is always Gumbel in the Gaussian case (or fixed-marginal setting, see Sections \ref{sec-maximum interpoint} and \ref{sec-sample coeff} below). Similarly, a non-Gumbel distribution can appear for the maximum interpoint distance in \cite{heiny2024maximum}. 
    
    \item The results in Theorem \ref{infty} below can be used to provide asymptotically exact thresholds for multiple testing in a Gaussian graphical model. 
\end{enumerate}

\noindent \textbf{Structure of this paper.}  Our main results on the maximum of the Gaussian field $\mathcal{G}_n$ are presented in Section~\ref{sec-main}. 
The three applications -- to the maximum interpoint distance, sample coefficients of equicorrelated populations, and multiple testing -- are given in Sections~\ref{sec-maximum interpoint}, \ref{sec-sample coeff}, and~\ref{sec-FWER}, respectively. 
The proofs are collected in Sections~\ref{sec-proof 1}, \ref{sec-proof 2}, and~\ref{sec-proof 3}, and some additional technical results can be found in Section~\ref{sec-technical}.

\section{Main results} \label{sec-main}
{To describe the limiting distribution of the maximum of the Gaussian field $\mathcal{G}_n$ in \eqref{sparse equi GF}, we need the sequences 
 \begin{align}
     c_n& :=\sqrt{2 \log n}  \quad \text{ and } \quad 
     d_n  := \sqrt{2 \log n} - \frac{\log \log n + \log 4 \pi }{2\sqrt{2\log n}}. \label{dn}
 \end{align}
 We assume that the correlation parameter $r=r_n$ depends on $n$, unless explicitly stated otherwise. Throughout this paper, all limits are for $n\to \infty$.}
 
 Our first result concerns the weakly dependent regime. 
\begin{theorem} \label{infty}
    Suppose  $(1-2r) \frac{\sqrt{\log n}}{\log \log n} \to \infty$ and recall the Gaussian field $\mathcal{G}_n$ in \eqref{sparse equi GF}. Then,
    \[
    \sqrt{2}c_n \lb \max_{1\leq i<j \leq n} G_{ij} - \sqrt{2}c_n + \frac{\log \lb 4\sqrt{\pi} c_n \rb}{\sqrt{2}c_n} \rb
    \]
    converges in distribution to the standard Gumbel law with CDF $\exp \lb - e^{-x} \rb$. 
\end{theorem}
{It is well-known that the normal distribution is in the maximum domain of attraction of the Gumbel law, which corresponds to the special case $r=0$  in Theorem \ref{infty} (see for instance \cite{embrechts:kluppelberg:mikosch:1997}). }Theorem \ref{infty} states that even under dependence the maximum of $\mc{G}_n$ behaves like that of i.i.d. standard normal random variables, as long as $(1-2r) \sqrt{\log n}/ \log \log n \to \infty$. Next, we describe the limiting distribution at the critical regime where this i.i.d.\ resemblance breaks down. 

Let $\la \zeta_k; k \geq 1 \ra$ be a sequence of i.i.d. standard exponential random variables and define
\begin{align} \label{eta}
\eta_i := -\log \lb \sum_{k=1}^i \zeta_k \rb\,, \qquad i\ge 1.
\end{align}
Then the set $\la \eta_i; i \geq 1 \ra$ forms a realization of a Poisson Point Process (P.P.P.) with intensity measure $e^{-x}dx$. We will use this equivalence interchangeably throughout the paper. 

{At the critical correlation regime, the limiting distribution of $\max_{1\leq i<j \leq n} G_{ij}$ can be expressed as the supremum of the points $\eta_i$ and i.i.d.\ normal variables, as described in the next result.}
\begin{theorem} \label{critical}
    Suppose $(1-2r)\log n \to \lambda \in (0,\infty)$ and recall the Gaussian field $\mathcal{G}_n$ in \eqref{sparse equi GF}. Then it holds 
    \begin{align} \label{transition}
        c_n \lb \max_{1\leq i<j \leq n} G_{ij} - \sqrt{2}d_n \rb \stackrel{d}{\to} \sup_{i<j} \lb \frac{\eta_i+ \eta_j}{\sqrt{2}} + \sqrt{2\lambda} \cdot  Z_{ij} \rb -\lambda,
    \end{align}
    where $\eta_i$'s are defined in \eqref{eta} and $Z_{ij}$'s are i.i.d. standard normal independent from $\eta_k$'s. 
\end{theorem}
Note that by Lemma \ref{well defined}, the supremum on the right-hand side of \eqref{transition} has an almost surely finite, non-degenerate distribution.
{In Theorem~\ref{critical}, the correlation parameter $r$ tends to $1/2$ at a specific rate. If we let $r\to 1/2$ at a faster rate, it turns out that the $Z_{ij}$'s disappear in the expression for the limiting distribution.}
\begin{theorem} \label{0}
    Suppose $(1-2r)\log n \to 0$. Then it holds 
    \begin{align*}
         c_n \lb \max_{1\leq i<j \leq n} G_{ij} - \sqrt{2}d_n \rb \stackrel{d}{\to} \frac{\eta_1+\eta_2}{\sqrt{2}}.
    \end{align*}
\end{theorem}

\begin{remark}{\em
Let us compare the scalings in Theorem~\ref{infty} versus Theorems~\ref{critical} and~\ref{0}, respectively. 
In all three cases, the scaling sequences are of the same order, but the second-order terms in the centering sequences differ. 
For Theorem~\ref{infty}, the second-order term in the centering is $\log \log n / \bigl(4\sqrt{\log n}\bigr)$, whereas for Theorems~\ref{critical} and~\ref{0} it is $\log \log n / \bigl(2\sqrt{\log n}\bigr)$.

This subtle difference can be explained by examining which terms contribute to the maximum. 
{Our proofs reveal that in Theorem~\ref{infty}} we have $\Theta(p^2)$ terms $G_{ij}$'s that contribute to the maximum on the same scale, whereas in Theorems~\ref{critical} and~\ref{0} only a set of size $\Theta(p)$ contributes significantly to the maximum. Here the notation $a_n = \Theta(b_n)$ means that $b_n/C \leq a_n \leq Cb_n$ for some universal constant $C>0$.
}\end{remark}


\section{Application I: Maximum interpoint distance} \label{sec-maximum interpoint}
The maximum interpoint distance is a classical object in stochastic geometry and statistics and has been studied extensively over the last few decades. 

Suppose $\bm{X}_1,\dots,\bm{X}_p$ are i.i.d.\ random vectors in $\mathbb{R}^n$, where each $\bm{X}_i = (x_{i1},\dots,x_{in})^\top$ has i.i.d.\ components with the same distribution as some generic random variable $\xi$, which we assume to be centered with unit variance. { Here $n$ denotes the dimension and $p$ the sample size. 
We adopt this convention to be consistent with \cite{heiny:kleemann:2024}. 
Note that \cite{tang2022asymptotic} uses $(n,p)$ to denote the sample size and the dimension, respectively.} 

The goal is to investigate the asymptotic distribution of
\begin{align} \label{Dn2}
D_n := \max_{1 \leq i < j \leq p} \| \bm{X}_i - \bm{X}_j \|_2
\end{align}
as $n,p \to \infty$.

This problem has a long history, and an incomplete list of results includes \cite{matthews1993asymptotic,henze1996limit,jammalamadaka2015asymptotic}; see also the references therein. While these early results were established in fixed-dimensional settings, recent advances in high-dimensional statistics motivate the need to understand the asymptotic distribution when the dimension diverges. 

The limiting distribution of $D_n$ is much less understood in this high-dimensional regime and to the best of our knowledge, it has only been studied in \cite{tang2022asymptotic,heiny2024maximum}. In both works, the authors show that $D_n$ is asymptotically Gumbel distributed when $p,n \to \infty$ and under the condition that 
$$\mb{E} \xi^4 \leq 5.$$ 
In what follows, we improve their results by dropping the above condition. Furthemore, we show that different limiting distributions appear if the law of $\xi$ is allowed to be $n$-dependent, that is, $\xi=\xi_n$. As a consequence, these limiting distributions may crucially depend on the quantity $\mathbb{E} \xi_n^4$.

\subsection{The fixed marginals case} We impose one of the following conditions on $\xi$:

\begin{align}
  \exists\,s>2,\ve>0:\quad 
  & \mathbb{E}\!\left[\,|\xi|^{2s+\ve}\,\right] <\infty.
  \tag{B1} \label{B1}
\\[0.8em]
  \exists\,\eta>0:\quad 
  &\mathbb{E}\!\left[e^{\eta |\xi|^2}\right] < \infty.
  \tag{B2} \label{B2}
\\[0.8em]
  \exists\,K<\infty:\quad 
  &\mb{P}(|\xi|\le K)=1.
  \tag{B3} \label{B3}
\end{align}
One can see that our conditions \eqref{B1}--\eqref{B3} are similar to those of Theorem~2.1 in \cite{heiny2024maximum}, except that we do not impose any restriction on the size of $\mathbb{E}\xi^4$. We refer to Remark~\ref{rem:thm2.1} for technical details. 

{Under either one of these assumptions, the maximum interpoint distance is asymptotically Gumbel distributed, provided that $p$ does not grow too fast relative to $n$.}

\begin{proposition} \label{maximum interpoint}
Under one of conditions \eqref{B1}--\eqref{B3}, we have, 
\begin{align} \label{scaling}
 \sqrt{2}c_p \lb \frac{D_n^2 -2n}{\sqrt{2n(1+ \mb{E} \xi^4)}} - \sqrt{2}c_p+\frac{\log \lb 4\sqrt{\pi} c_p \rb}{\sqrt{2}c_p} \rb\ 
\end{align}
converges in distribution to the standard Gumbel law with CDF $\exp \lb - e^{-x} \rb$, where $c_p$ is defined in \eqref{dn} with $n$ being replaced by $p$, and  $p$ satisfies 
\begin{itemize}
\item[(i)] $p = O(n^{(s-2)/4})$ if \eqref{B1} holds.
\item[(ii)] $p = \exp\!\big( o(n^{1/5}) \big)$  if \eqref{B2} holds.
 \item[(iii)] $p = \exp\!\big( o \lb \frac{n^{1/3}}{(\log n)^{2/3}} \rb \big)$  if \eqref{B3} holds.
\end{itemize}
\end{proposition}
{ Note that the scaling in \eqref{scaling} is exactly the same as in \cite{heiny:kleemann:2024}.}

\noindent \textbf{Proof of Proposition \ref{maximum interpoint}.} 
Let $\mathcal{G}_n= \la G_{ij}: 1\leq i<j \leq p \ra$ be the Gaussian field in \eqref{sparse equi GF} with correlation parameter 
$r = (\mb{E} \xi^4 -1)/(2 (\mb{E} \xi^4 + 1 ))$. Since \[
1-2r = \frac{2}{\mb{E} \xi^4 +1} >0, 
\]
an application of Theorem \ref{infty} yields that
$$
\mb{P} \lb
 \sqrt{2}c_p \Big( \max_{1 \leq i <j \leq p} G_{ij} 
 - \sqrt{2}c_p+\frac{\log \lb 4\sqrt{\pi} c_p \rb}{\sqrt{2}c_p} \Big)\
\leq x \rb\to \exp \lb -e^{-x} \rb\,, \qquad x\in \mathbb{R}.
$$
Therefore, it is suffices to check that
\begin{align} \label{Gaussian approximation}
\sup_{t \in \mb{R}} \Big| \mb{P} \lb \frac{D_n^2 -2n}{\sqrt{2n(1+ \mb{E} \xi^4)}} \leq t \rb - \mb{P} \Big( \max_{1 \leq i <j \leq p} G_{ij} \leq t \Big)   \Big| \to 0
\end{align}
if either one of \eqref{B1}--\eqref{B3} holds with the corresponding requirement on $p$ as in the statement of Proposition \ref{maximum interpoint}.
To verify \eqref{Gaussian approximation}, observe that
\begin{align*}
        \mb{P} \lb D_n^2 \leq t  \rb &= \mb{P} \lb   \sum_{k=1}^n \lb x_{ki} - x_{kj} \rb^2 \leq t : 1\leq i<j \leq p \rb
\end{align*}
{is the probability of a sum of i.i.d. random vectors staying in the hyperrectangle $[-\infty,t]^{p(p-1)/2}$. For high-dimensional normal approximation of the sum, we calculate the covariance structure of the vectors: }
\begin{align*}
\mb{E} \lb x_{12}-x_{13} \rb^2 &= 2, \\
 \mbox{Var} \left[   \lb x_{12} - x_{13} \rb^2   \right] &=  2\lb 1 + \mb{E} \xi^4 \rb, \\
  \mbox{Corr} \left[ \lb x_{12} - x_{13} \rb^2,\lb x_{12} - x_{14} \rb^2 \right] &= \frac{\mb{E} \xi^4 -1}{2\lb \mb{E} \xi^4 + 1 \rb}=r.
\end{align*}
We then obtain \eqref{Gaussian approximation} by applying Theorem~2.1(b) of \cite{Koike} under condition~\eqref{B1} with $q = s + \ve/2$ (see also the bound on the right-hand side of \eqref{hclt-Dn} below where we derive the bounds explicitly). Under condition \eqref{B2}, we apply Corollary~2.1 in \cite{Koike} to deduce \eqref{Gaussian approximation}.

Finally, under \eqref{B3}, we get \eqref{Gaussian approximation} due to the condition $\log p = o \lb n^{1/3}/ (\log n)^{2/3} \rb$ by Corollary 2.1 in \cite{chernozhukov2023nearly}. Note that such a result is applicable because the smallest eigenvalue of the covariance matrix of 
\[
\la  \lb x_{1i} - x_{1j} \rb^2 -2  \ra_{1\leq i<j \leq p} 
\]
is strictly away from $0$. This will be verified in Lemma \ref{positive definite} below, completing the proof of \eqref{Gaussian approximation}. $\hfill$ $\square$

\begin{remark} \label{rem:thm2.1} {\em
A few remarks about the assumptions of Proposition \ref{maximum interpoint} are in place.
\begin{itemize}
\item[(a)] Condition~\eqref{B2} coincides with condition (B3) in \cite{heiny2024maximum} for $r=1$ and condition \eqref{B3} is similar to condition (B4) in \cite{heiny2024maximum}, up to the mild logarithm term on $n$. Condition~\eqref{B1} differs from condition (B1) in \cite{heiny2024maximum} only by an additional $\ve$ factor. We would like to point out that a more careful inspection of the proof of \eqref{Gaussian approximation} reveals that condition \eqref{B1} can be replaced by $\mathbb{E}\!\left[\,|\xi|^{2s} (\log|\xi|)^M\,\right] <\infty$ for some large constant $M$. We retain this minimal sub-optimality in order to maintain a simpler analysis. 
\item[(b)] We believe that it should be possible to reproduce the results under the same assumptions as in Theorem~2.1 of \cite{heiny2024maximum}, but without any restriction on the fourth moment, by adapting the proof of Theorem~\ref{infty} so that it remains compatible with the large deviation calculations. We do not pursue this direction since our goal is to present a unified treatment of several related problems.
\item[(c)] A direct application of high-dimensional normal approximation, as performed in the proof of \eqref{Gaussian approximation}, would lead to sub-optimal dimension dependence in the case where $\xi$ is sub-Weibull, which corresponds to condition (B2) in \cite{heiny2024maximum}. However, the sub-optimality is very mild in the three cases \eqref{B1}--\eqref{B3} we consider here.
\end{itemize}
}\end{remark}

\subsection{The $n$-dependent marginals case} In this subsection, we will use Theorems \ref{critical} and \ref{0} to demonstrate that new limiting distributions can arise when $\xi = \xi_n$ is allowed to be $n$ dependent and diverges. To this end, we will focus on the simple high-dimensional linear regime that $p \asymp n$ for simplicity. It is worth mentioning that the arguments below can be adapted to other regimes and moment conditions in a relatively straightforward manner. 

We assume the following conditions: 
\begin{align}
 &0< \inf_{n \geq 1} \frac{p}{n} \leq \sup_{n \geq 1}   \frac{p}{n}  < \infty \ \quad \text{and} \quad n,p \to \infty.
  \tag{C1} \label{C1}
\\[0.8em]
\exists \ve>0:\quad  &  \mb{E} \lb |\xi_n|^{12+\ve} \rb <\infty,\ \forall n \geq 1 .
\tag{C2} \label{C2}
\\[0.8em]
\exists C_\ve>0: \quad &\sup_{n \geq 1}  \frac{\mb{E} \lb |\xi_n|^{12+\ve} \rb}{\left[ \mb{E} \lb \xi_n^{4} \rb \right]^{C_\ve}}  < \infty.
\tag{C3} \label{C3}
\end{align}
Condition \eqref{C1} requires that $p$ grows at most linearly with $n$. 
Condition \eqref{C2} assumes that $\xi_n$ has a finite $(12+\ve)$-moment for every $n$ (though it may diverge as $n \to \infty$). 
Finally, Condition \eqref{C3} is a type of reverse H\"older inequality, which states that the $(12+\ve)$-moment of $\xi_n$ grows at most polynomially in terms of its fourth moment. 
This condition is needed to guarantee the validity of the Gaussian approximation argument.  

Intuitively, Conditions \eqref{C2} and \eqref{C3} imply that the mass of the distribution of $\xi_n$ can be gradually pushed toward infinity, but at a scale where its higher moments grow at most polynomially fast. 

One such example of $\xi_n$ is the mixture distribution
\begin{align} \label{example}
\xi_n \stackrel{d}{=} \delta_n \mbox{Unif} \left[ -1,1 \right] + (1-\delta_n) \mbox{Unif}\left[ (-2e_n,-e_n) \cup (2e_n,e_n)  \right]\,,
\end{align}
where $\delta_n:= 2/\lb 7e_n^2-1 \rb$ and $e_n \geq 10$.
Note that $\xi_n$ chosen above is centered, unit variance, and satisfies 
\[
 \mathbb{E}\xi_n^{2k}
= \frac{1}{2k+1}\Big[\,1 + \delta_n\big( (2^{\,2k+1}-1)\, e_n^{\,2k} - 1 \big) \Big] 
\]
for every fixed $k \in \mb{N}$. Clearly, for every fixed $k>1$, the $2k$-th moment of $\xi_n$ diverges if $e_n$ diverges.
It is easy to check that this mixture distribution satisfies \eqref{C2} and \eqref{C3} with $\ve=2$ and $C_\ve=5$ for every sequence $\la  e_n; n \geq 1 \ra$. 

{As a consequence of Theorems \ref{critical} and \ref{0}, we can show under \eqref{C1}--\eqref{C3} the maximum interpoint distance can have a non-Gumbel limiting distribution.}

\begin{proposition} \label{maximal interpoint phase transition}
    Suppose \eqref{C1}--\eqref{C3} hold and recall the sequences $c_n$ and $d_n$ from \eqref{dn}. Assume that 
   $(\log p)/ \mb{E}  \xi_n^4   \to \lambda$ 
   for some $\lambda \in [0,\infty)$, where in the case $\lambda=0$ we additionally require that
   \begin{align} \label{growth bound}
       \mb{E} \lb  \xi_n^4 \rb =O \lb (\log p)^A \rb
   \end{align}
   for some constant $A>1$.  Then it holds
    \begin{align*}
      c_p \lb \frac{D_n^2 -2n}{\sqrt{2n(1+ \mb{E} \xi_n^4)}} - \sqrt{2}d_p\rb 
    \stackrel{d}{\to} & \sup_{i<j} \lb \frac{\eta_i+ \eta_j}{\sqrt{2}} + 2\sqrt{\lambda} \cdot  Z_{ij} \rb -2\lambda
    \end{align*}
    where $\eta_i$'s are defined in \eqref{eta} and $Z_{ij}$'s are i.i.d. standard normal random variables independent from $\eta_i$'s, and $c_p$ and $d_p$ are defined as in \eqref{dn}.
\end{proposition}
The limiting random variable in the case $\lambda=0$ in Proposition \ref{maximal interpoint phase transition} simplifies to 
\[
\sup_{i<j}  \frac{\eta_i+ \eta_j}{\sqrt{2}} =\frac{\eta_1+\eta_2}{\sqrt{2}},
\]
where we used the fact that $\la \eta_k; k \geq 1 \ra$ forms a decreasing sequence.
\begin{remark}{\em
Note that Proposition~\ref{maximal interpoint phase transition} does not cover the case $\lambda=\infty$,  for which one would need the stronger condition 
$\big(\sqrt{\log p}\big)/ \mb{E}   \xi_n^4   \to \infty$.
If this condition holds, then the limiting distribution {is the standard Gumbel law as in Proposition~\ref{maximum interpoint}(i)}. The gap between the scales $\sqrt{\log p}$ and $\log p$ arises from the gap in Theorems~\ref{infty} and~\ref{critical}. We conjecture that Theorem~\ref{infty} remains valid up to the scale $\log p$, although the analysis based on Poisson approximation is unlikely to extend that far.
}\end{remark}

\noindent \textbf{Proof of Proposition \ref{maximal interpoint phase transition}.} Let us first check that \eqref{Gaussian approximation} is still valid under conditions  \eqref{C1}--\eqref{C3}. Observe that
\[
\mb{E} [\lb x_{12}-x_{13} \rb^2]= 2 \quad \text{ and } \quad   \Var \left[   \lb x_{12} - x_{13} \rb^2   \right] =  2 \lb 1 + \mb{E} \xi_n^4 \rb.
\]
Let $\mathcal{G}_n= \la  G_{ij}; 1\leq i<j\leq p \ra$ be the Gaussian field in \eqref{sparse equi GF} with correlation parameter 
\[
r_n:= \frac{\mb{E} \xi_n^4 -1}{2\lb \mb{E} \xi_n^4 + 1 \rb}.
\]
By using the high-dimensional central limit theorem (see Theorem 2.1 in \cite{Koike}), we obtain
\begin{align}
    \sup_{t \in \mb{R}} &\left| \mb{P} \lb \frac{D_n^2 - 2n}{\sqrt{2n\lb 1 + \mb{E} \xi_n^4 \rb}}  \leq t  \rb  - \mb{P} \lb \max_{1\leq i<j \leq p}  G_{ij} \leq t \rb  \right|  \nonumber \\
\lesssim_{\ve}  &  \lb \frac{B_{n1}^2 (\log p)^5}{n}  \rb^{1/6} + \lb \frac{B_{n2}^2 \lb \log p  \rb^{3-2/q}}{n^{1-2/q}}  \rb^{1/3}  \,,\label{hclt-Dn}
\end{align}
where $q:=6+\ve/4$ and 
\begin{align*}
    B_{n1}^2& := \frac{\mb{E} \left[ \big( \lb x_{12} -x_{13} \rb^2-2 \big)^4 \right]}{\left[ 1 + \mb{E} \xi_n^4  \right]^{2}} \leq \frac{128 + 8\mb{E} \lb x_{12} - x_{13} \rb^8}{\left[ 1 + \mb{E} \xi_n^4   \right]^{2}}
    \leq \frac{128 \left[ 1 + 16\mb{E} \xi_n^8 \right]}{\left[ 1 + \mb{E} \xi_n^4 \right]^{2}};\\
    B_{n2}& := \frac{1}{\sqrt{1+ \mb{E} \xi_n^4}}\Big\| \max_{1\leq i<j \leq p} \left| \lb x_{1i} -x_{1j} \rb^2 -2 \right| \Big\|_{L_q}.
\end{align*}
By using assumption \eqref{C2} and H\"older's inequality, we conclude that $B_{n1}^2= O \lb (\log p)^{C_1} \rb$ for some constant $C_1$ depending only on $\ve$. For the term $B_{n2}$, note that assumption \eqref{C2} implies that
\begin{align*}
    \Big\| \max_{1\leq i<j \leq p} \left| \lb x_{1i} -x_{1j} \rb^2 -2 \right| \Big\|_{L_q} &\leq p^{2/q} \times \mb{E} \left[  \left| \lb x_{1i} -x_{1j} \rb^2 -2 \right|^q \right]^{1/q} \\
    &= O \lb p^{2/q} (\log p)^{C_2} \rb
\end{align*}
for some constant $C_2>0$ not depending on $n$. 

Thus, to deduce that \eqref{hclt-Dn} tends to $0$, we only need to check that there exists $C_3>0$ so that
\[
\frac{B_{n2}^2}{n^{1-2/q}} \ll n^{-C_3}.
\]
This is true since 
\[
\frac{p^{4/q}}{n^{1-2/q}} = O \lb n^{6/q-1} \rb = O\lb n^{-\ve/(24+\ve)} \rb.
\]
The proof of Proposition \ref{maximal interpoint phase transition} is completed by invoking Theorems \ref{critical}, \ref{0} and using \eqref{hclt-Dn} together with the fact that 
\[
(1-2r_n)\log p =  \frac{2 \log p}{1+ \mb{E} \xi_n^4} \to 2\lambda.
\]
$\hfill$$\square$

\section{Application II: Sample coefficients of equicorrelated populations} \label{sec-sample coeff}

The study of extremes of  sample covariance and sample correlation coefficients can be traced back to the seminal work \cite{Jiang05}. Similar results and techniques have since been used in a number of different problems ranging from testing covariance structures and independence \cite{Cai17,bastian2024testing}, two-sample tests \cite{Cai-Liu-Xia}, to community detection \cite{hu2021using}. 

Let us briefly recall the problem settings and related results in what follows. We are given i.i.d. data points $\bm{X}_1, \dots,\bm{X}_n$ with $\bm{X}_i=\lb x_{i1},x_{i2},\dots,x_{ip} \rb^\top$. Recall the sample covariance coefficients $\hat{r}_{ij}$ and sample correlation coefficients $\hat{\rho}_{ij}$ defined in \eqref{sample covariance} and \eqref{pearson}, respectively. Their maxima are given by 
\begin{align}
    R_n &:= \max_{1\leq i<j \leq p} \hat{r}_{ij} \qquad \text{ and } \qquad 
    M_n :=  \max_{1\leq i<j \leq p} \hat{\rho}_{ij}. \label{Mn}
\end{align}
There is a long line of literature on the asymptotic distributions of $R_n$ and $M_n$ in the independent setting, where the coordinates $x_{ki}$ are i.i.d. The asymptotic Gumbel law of $R_n$ and $M_n$ established in \cite{Jiang05} was extended in \cite{Zhou07} to a more general framework that only requires finiteness of the sixth moment of the $x_{ij}$'s and the condition $p = O(n)$. In subsequent works \cite{Li1,Li2,Li3}, necessary and sufficient conditions were derived for the Gumbel law to hold. With respect to the dependence between $p$ and $n$, optimal dimension-dependence results were obtained in \cite{Liu16}, \cite{ShaoZhou}, \cite{heiny:mikosch:yslas:2021} and \cite{heiny:kleemann:2024}. Results regarding the banded covariance matrices can be found in \cite{boucher2025largest}.

The asymptotic distributions of $R_n$ and $M_n$ in high dimensions are highly intricate. High dimensionality, dependence structures and heavy-tails can disrupt the classical Gumbel law. High dimensionality typically breaks the Gumbel limit by inducing clustering in the extremes (often quantified by the extremal index), whereas dependence structures (even when being simple) can alter the limiting distribution entirely, leading to asymptotic laws that no longer belong to the classical extreme-value families of Gumbel, \Frechet or Weibull type.

Results in the dependent settings are much scarcer. To the best of our knowledge, \cite{Jiang19} and \cite{jiang2023asymptotic} are the only works that investigate this case. In particular, \cite{Jiang19} considers the Gaussian population with an equicorrelated correlation structure, while \cite{jiang2023asymptotic} studies the auto-regressive correlation structure. Both papers show that in these settings the asymptotic distributions of $R_n$ and $M_n$ are no longer Gumbel.

We now describe the results of \cite{Jiang19} in more detail. Suppose the distribution of $\bm{X}_i = (x_{i1}, x_{i2}, \dots, x_{ip})^\top$ admits the representation
\begin{align} \label{latent variables}
  x_{ik} := \sqrt{\rho}\,\xi_i + \sqrt{1-\rho}\,\xi_{ik}  
\end{align}
where $\la \xi_i, \xi_{ik} : 1 \leq i \leq n,\, 1 \leq k \leq p \ra$ are i.i.d.\ with mean $0$ and unit variance. For simplicity, denote their common law by $\xi$. It was shown in \cite{Jiang19} that when $\xi$ is standard normal and $\log p = o(n^{1/3})$ the following hold:  

\textit{(i)} For $R_n$ defined in \eqref{Mn}, under the condition $\limsup \rho_n < 1/2$, $R_n$ is asymptotically Gumbel, normal, or a mixture of independent Gumbel and normal, depending on whether $\rho \sqrt{\log p}$ converges to $0$, $\infty$, or a finite nonzero constant, respectively.  

\textit{(ii)} For $M_n$ defined in \eqref{Mn}, under the condition $\limsup \rho_n < 1$, $M_n$ is asymptotically Gumbel, normal, or a mixture of independent Gumbel and normal, depending on whether $\rho \sqrt{\log p}$ converges to $0$, $\infty$, or a finite nonzero constant, respectively.  

The proofs in \cite{Jiang19} are highly technical and involved. 
In this section, we recover the results for both $R_n$ and $M_n$ with very simple proofs, and moreover without the $1/2$ restriction on the range of $\rho$. We further show that different limiting distributions can arise when the distribution of $\xi$ is allowed to depend on $n$.

\subsection{The fixed marginal case} \label{sec-fixed marginal pearson}
We will first show the results under the stronger assumption that $\log p =o(n^{1/5})$. While this condition is stronger than $\log p =o(n^{1/3})$, the proofs are much shorter and more intuitive, and more importantly, it requires no restriction on the range of $\rho$. In what follows, we put $\kappa=\mb{E} \xi^4$ and make the additional assumption that the distribution of $\xi$ is sub-Gaussian, that is,
$\mb{E} \exp \lb t\xi^2 \rb < \infty$ 
for some $t>0$.

\begin{proposition} \label{limit Rn}
Suppose $\log p = o(n^{1/5})$, $\limsup \rho_n < 1$ and recall $R_n$ from \eqref{Mn}. Define
\[
R_n^* := \frac{R_n}{\sqrt{n}} - \rho \sqrt{n} 
        - \left( 2\sqrt{\log p} - \frac{\log \log p}{4\sqrt{\log p}} \right)\sqrt{1-\rho^2}.
\]
Then the following statements hold.
\begin{itemize}
    \item[(i)] If $\rho \sqrt{\log p} \to 0$,  we have
    \[
    2\sqrt{\log p}\,R_n^* \;\;\stackrel{d}{\to}\; G_1\,,
    \]
    where $G_1$ has the Gumbel law with CDF $x \mapsto \exp\!\left( -\tfrac{1}{4\sqrt{2\pi}} e^{-x} \right)$.

    \item[(ii)] If $\rho \sqrt{\log p} \to \lambda \in (0,\infty)$,  we have
    \[
    \tfrac{1}{\rho} R_n^* \;\;\stackrel{d}{\to}\; \tfrac{1}{2\lambda}G_1 + N(0,\kappa-1),
    \]
    where $G_1$ and $N(0,\kappa-1)$ are independent.

    \item[(iii)] If $\rho \sqrt{\log p} \to \infty$,  we have
    \[
    \tfrac{1}{\rho} R_n^* \;\;\stackrel{d}{\to}\; N(0,\kappa-1).
    \]
\end{itemize}
\end{proposition}
Note that, in contrast to \cite{Jiang19} and \cite{heiny2024maximum}, we do not impose the condition $\limsup \rho_n < 1/2$. Moreover, with $\kappa=3$, we recover the results in \cite{Jiang19}, which corresponds to the Gaussian setting.

\noindent \textbf{Proof of Proposition \ref{limit Rn}.}  
Let us first get rid of the sample mean in \eqref{sample covariance}. We will show that
\begin{align} \label{centering}
\frac{R_n}{\sqrt{n}} = \max_{1\leq i<j \leq p} \frac{\sum_{k=1}^n x_{ki}x_{kj}}{\sqrt{n}} + O_{\mb{P}}\lb \frac{\log p}{\sqrt{n}}  \rb.  
\end{align}
To see this, write 
\[
\frac{1}{\sqrt{n}}\sum_{k=1}^{n} \lb x_{ki} - \bar{x}_i \rb \lb x_{kj} - \bar{x}_j \rb = \frac{1}{\sqrt{n}}\sum_{k=1}^{n} x_{ki}x_{kj} - \sqrt{n} \cdot \bar{x}_i \bar{x}_j.
\]
By using either Lemma 2 in \cite{jiang2023asymptotic} or Bernstein's inequality, we get
\[
\max_{1\leq i\leq p} \sqrt{n} |\bar{x}_i| =  O_{\mb{P}} \lb \sqrt{\log p} \rb.
\]
Thus, \eqref{centering} is justified.
Straightforward calculation using \eqref{latent variables} gives
\begin{align*}
    \mb{E}(x_{1i}x_{1j} x_{1k} x_{1l}) = \begin{cases} 
          (\kappa-1)\rho^2, & \quad \big|\la i,j \ra \cap \la k,l \ra\big| = 0; \\
          \rho+(\kappa-2)\rho^2, & \quad\bigm| \la i,j \ra \cap \la k,l \ra \big| = 1;\\
          1+ (\kappa-2)\rho^2, & \quad\bigm| \la i,j \ra \cap \la k,l \ra \big| = 2.
       \end{cases}
\end{align*}
Next, by using the high-dimensional central limit theorem (Theorem 2.1 in \cite{Koike}), we get $\log p = o(n^{1/5})$ that
\begin{align} \label{hclt-Rn}
    \sup_{t \in \mb{R}} \left| \mb{P} \lb \max_{1\leq i<j \leq p} \frac{\sum_{k=1}^n x_{ki}x_{kj} - n\rho}{\sqrt{n(1+(\kappa-2)\rho^2)}} \leq t \rb - \mb{P} \lb \max_{1\leq i<j\leq p} Z_{ij} \leq t \rb  \right| \to 0\,,
\end{align}
where $\la Z_{ij}; 1\leq i<j \leq p \ra$ is the centered Gaussian field with correlation structure 
\begin{align*}
    \mb{E}(Z_{ij} Z_{kl}) = \begin{cases} 
          \frac{(\kappa-1)\rho^2}{1+(\kappa-2)\rho^2}, &  \big| \la i,j \ra \cap \la k,l \ra \big| = 0; \\
          \frac{\rho+(\kappa-2)\rho^2}{1+(\kappa-2)\rho^2}, & \bigm| \la i,j \ra \cap \la k,l \ra \big| = 1;\\
          1, & \bigm| \la i,j \ra \cap \la k,l \ra \big| = 2.
       \end{cases}
\end{align*}
From this, one can write 
\begin{align} \label{dist representation}
    \max_{1\leq i<j \leq p} Z_{ij} \stackrel{d}{=} \sqrt{\frac{(\kappa-1) \rho^2}{1+(\kappa-2)\rho^2}} \cdot G + \sqrt{\frac{1-\rho^2}{1+(\kappa-2)\rho^2}} \cdot \max_{1 \leq i<j \leq p} G^{*}_{ij}\,,
\end{align}
where $G$ is a standard normal distribution independent from $G_{ij}$'s and the Gaussian field $\la G_{ij}; 1\leq i<j\leq p \ra$ is as in \eqref{sparse equi GF} with correlation parameter
$r = \frac{\rho}{1+\rho}$.
In view of the assumption on $\rho$, it follows that
\[
1-2r = \frac{1-\rho}{1+\rho}
\]
is strictly greater than zero.  Consequently, $(1-2r)\sqrt{\log p}/ \log \log p \to \infty$. The proof is completed by combining \eqref{centering}, \eqref{hclt-Rn}, \eqref{dist representation} and then applying Theorem \ref{infty} together with some simple (but tedious) algebra. $\hfill$ $\square$

Regarding the largest sample correlation cofficients, we can show the following result.
\begin{proposition} \label{limting Mn}
    Suppose $\log p = o(n^{1/5})$, $\limsup \rho_n < 1$ and recall $M_n$ from \eqref{Mn}. Define
    \[
    M^*_n:= \sqrt{n}M_n - \rho \sqrt{n} - (1-\rho) \sqrt{1+ 2\rho + \frac{\kappa-5}{2}\rho^2} \cdot \lb 2\sqrt{\log p} - \frac{\log \log p }{4 \sqrt{\log p}}\rb.
    \]
  Then the following statements hold.
\begin{itemize}
    \item[(i)] If $\rho \sqrt{\log p} \to 0$,  we have
    \[
    2\sqrt{\log p}\,M_n^* \;\;\stackrel{d}{\to}\; G_1\,,
    \]
    where $G_1$ has the Gumbel law with CDF $x \mapsto \exp\!\left( -\tfrac{1}{4\sqrt{2\pi}} e^{-x} \right)$.

    \item[(ii)] If $\rho \sqrt{\log p} \to \lambda \in (0,\infty)$,  we have
    \[
    \tfrac{1}{\rho} M_n^* \;\;\stackrel{d}{\to}\; \tfrac{1}{2\lambda}G_1 + N(0,\kappa-1),
    \]
    where $G_1$ and $N(0,\kappa-1)$ are independent.

    \item[(iii)] If $\rho \sqrt{\log p} \to \infty$,  we have
    \[
    \tfrac{1}{\rho(1-\rho)} M_n^* \;\;\stackrel{d}{\to}\; N(0,\kappa-1).
    \]
\end{itemize}
\end{proposition}
The condition $\limsup \rho_n <1$ imposed in Proposition \ref{limting Mn} is the same as in \cite{Jiang19}, which is to prevent the sample correlation coefficients from collapsing to $1$ too quickly. It is easy to check that for $\kappa=3$, we recover the results in \cite{Jiang19}. 
\smallskip

\noindent \textbf{Proof of Proposition \ref{limting Mn}.}   By a minor modification of Step 1 in the proof of Theorem 1 in \cite{jiang2023asymptotic}, we have
\begin{align} \label{linearize}
    \max_{1\leq i<j \leq p} \sqrt{n}\hat{\rho}_{ij} -\rho\sqrt{n} = \max_{1\leq i<j\leq p} \la \frac{1}{\sqrt{n}}\sum_{k=1}^n x_{ki}x_{kj} - \frac{\rho}{2}\lb x_{ki}^2 + x_{kj}^2 \rb \ra + O_{\mb{P}}\lb \frac{\log p}{\sqrt{n}}  \rb.
\end{align}
For $k=1,\ldots,n$, put
\[
T^{(k)}_{ij}:= x_{ki}x_{kj} - \frac{\rho}{2} \lb x_{ki}^2+x_{kj}^2 \rb.
\]
Some moment calculations using \eqref{latent variables} give
\begin{align*}
\mb{E} \lb T_{ij}^{(1)} T_{kl}^{(1)} \rb = \begin{cases}
    (1-\rho)^2 \lb 1 + 2\rho + \frac{\rho^2}{2}(3\kappa-7) \rb,& \bigm| \la i,j \ra \cap \la k,l \ra \big| = 2; \\
    (1-\rho)^2 \bigg[ \rho + \lb \frac{5\kappa-9}{4} \rb \rho^2 \bigg],& \bigm| \la i,j \ra \cap \la k,l \ra \big| = 1; \\
    (1 - \rho)^2 (\kappa-1) \rho^2,& \bigm| \la i,j \ra \cap \la k,l \ra \big| = 0.
\end{cases}
\end{align*}
By the high-dimensional central limit theorem (Theorem 2.1 in \cite{Koike}), we have for $\log p =o(n^{1/5})$
\begin{align} \label{hclt Tn}
     \sup_{t \in \mb{R}} \left| \mb{P} \lb \max_{1\leq i<j \leq p} \frac{\sum_{k=1}^n T^{(k)}_{ij}}{(1-\rho)\sqrt{n(1 + 2\rho + \frac{\rho^2}{2}(3\kappa-7)}} \leq t \rb - \mb{P} \lb \max_{1\leq i<j\leq p} Z_{ij} \leq t \rb  \right| \to 0\,,
\end{align}
where $\la Z_{ij}; 1\leq i<j \leq p \ra$ is the centered Gaussian field with correlation structure 
\begin{align*}
    \mb{E}(Z_{ij} Z_{kl}) = 
           \begin{cases} 
         \frac{(\kappa-1)\rho^2}{1 + 2\rho + \frac{\rho^2}{2}(3\kappa-7)}, &  \big|\la i,j \ra \cap \la k,l \ra \big|= 0; \\
          \frac{\rho + \lb \frac{5\kappa-9}{4} \rb \rho^2}{1 + 2\rho + \frac{\rho^2}{2}(3\kappa-7)}, & \bigm| \la i,j \ra \cap \la k,l \ra \big| = 1;\\
          1, & \bigm| \la i,j \ra \cap \la k,l \ra \big| = 2.
       \end{cases}
\end{align*}
Thus, we further have the decomposition 
\begin{align} \label{dist representation2}
    \max_{1\leq i<j \leq p} Z_{ij} \stackrel{d}{=} \sqrt{\rho_1}G + \sqrt{1-\rho_1} \cdot \max_{1\leq i<j \leq p} G_{ij} \,,
\end{align}
where $\rho_1:=  \rho^2(\kappa-1)(1+ 2\rho + (\rho^2/2)(3\kappa-7))^{-1}$, $G$ is a standard normal distribution independent from $G_{ij}$'s and the Gaussian field $\la G_{ij}; 1\leq i<j\leq p \ra$ is as in \eqref{sparse equi GF} with correlation parameter
\[
r = \frac{\rho + \frac{\kappa-5}{4} \rho^2}{1+ 2\rho + \frac{\kappa-5}{2} \rho^2}.
\]
Using the assumption on $\rho$ and the fact that $\kappa \ge 1$, we deduce that
\[
1-2r = \frac{1}{1+2\rho + \frac{\kappa-5}{2} \rho^2}
\]
is strictly greater than zero. Consequently, $(1-2r)\sqrt{\log p}/ \log \log p \to \infty$.
The proof is completed by combining \eqref{linearize}, \eqref{hclt Tn}, \eqref{dist representation2} and applying Theorem \ref{infty} together with some simple (but tedious) algebra. $\hfill$ $\square$

\subsection{The $n$-dependent marginals case}
Recall $R_n$ in \eqref{Mn}. We now show that if $\xi = \xi_n$ is allowed to depend on $n$, the limiting distribution of $R_n$ can change drastically, depending on the magnitude of $\kappa_n := \mathbb{E}[\xi_n^4]$. { We will investigate the behavior of $R_n$ under two scenarios: $\xi_n$ behaves like a heavy-tailed distribution with masses shifted toward infinity, and when $\xi_n$ behaves like a Rademacher distribution with masses shifted toward two points. In both scenarios, the behaviors of $R_n$ are quite different than in Section \ref{sec-fixed marginal pearson}.}

Assume \eqref{C1}--\eqref{C3}. We will consider the following two regimes for $\kappa_n$: 

\begin{itemize}
    \item[(i)] $\kappa_n/(\log p) \to \lambda_1 \in(0,\infty) $ and $\rho_n\log p \to \lambda_2  { \in(0,\infty)}$ as $n\to \infty$.
    \item[(ii)] $(\kappa_n-1)(\log p)^2 \to \lambda_1 \in (0,\infty)$ and $(1-\rho_n)\log p \to \lambda_2 \in (0,\infty)$. 
\end{itemize}
An example of a distribution that satisfies conditions \eqref{C1}--\eqref{C3} and (i) is given in \eqref{example}  with 
$e_n= \Omega(\sqrt{\log p})$. 

An example of a distribution that satisfies conditions \eqref{C1}--\eqref{C3} and (ii) is the distribution
\begin{align*}
    \xi_n = \begin{cases}
        \sqrt{\frac{(\log p)^2 + \lambda_1}{ (\log p)^2 }}, \ &\text{with probability} \ \frac{(\log p)^2}{2(\lambda_1 + (\log p)^2)}; \\
        0, \ &\text{with probability} \ \frac{\lambda}{\lambda + (\log p)^2}; \\
        -\sqrt{\frac{(\log p)^2 + \lambda_1}{ (\log p)^2 }}, \ &\text{with probability} \ \frac{(\log p)^2}{2(\lambda_1 + (\log p)^2)}.
    \end{cases}
\end{align*}
Indeed, one can check that 
\[
\mb{E} \xi_n =0,\quad \ \mb{E} \xi_n^2=1,\quad \ \mb{E} \xi_n^4= 1+\frac{\lambda_1}{(\log p)^2},\quad \ \mb{E} \xi_n^{14} =  \lb 1+\frac{\lambda_1}{(\log p)^2} \rb^6\,.
\]
{More generally, condition (ii) implies that $\xi_n$ converges in distribution to a Rademacher variable. To see this, note that (ii) requires $\kappa_n\to 1$ which in turn implies that $\P(|\xi_n|\le 1)\to 1$. Since $\xi_n$ is mean zero and unit variance, one concludes that $\xi_n$ converges to Rademacher distribution.}

The following result describes the limiting distribution of $R_n$ in the $n$-dependent case.
\begin{proposition} \label{new distributions Rn}
Assume \eqref{C1}--\eqref{C3}.
\begin{itemize}
    \item[(i)] In regime (i),
    \begin{align} \label{regime i}
    2\sqrt{\log p}\left[ \frac{R_n}{\sqrt{n}} - \rho\sqrt{n} 
    - \sqrt{1-\rho^2}\left( 2\sqrt{\log p} - \frac{\log \log p}{4\sqrt{\log p}} \right)\right] \cid
    N(0,2\lambda_1\lambda_2^2) + G_1,
    \end{align}
    where $G_1$ is the Gumbel law in Proposition~\ref{limit Rn}, independent of the normal term.

    \item[(ii)] In regime (ii),
    \begin{align} \label{regime ii}
    \log p\!\left[ \frac{R_n}{\sqrt{n}} - \rho\sqrt{n} 
    - \sqrt{1-\rho^2}\!\left( 2\sqrt{\log p} - \frac{\log \log p + \log(4\pi)}{2\sqrt{\log p}} \right)\right]
    \end{align}
    converges in distribution to 
    \[
    N(0,\lambda_1) + \sqrt{\lambda_2}\,
    \sup_{i<j}\la \frac{\eta_i+\eta_j}{\sqrt{2}} 
    + \sqrt{\lambda_2}\,Z_{ij} - \frac{\lambda_2}{2} \ra,
    \]
    where $\eta_i$ are as in \eqref{eta}, $Z_{ij}\sim N(0,1)$ are i.i.d., and the normal component is independent of the rest.
\end{itemize}
\end{proposition}
\noindent \textbf{Proof of Proposition \ref{new distributions Rn}.} One can check that \eqref{hclt-Rn} still goes through under conditions \eqref{C1}--\eqref{C3}. 
The desired result then follows from applying Theorem \ref{critical} to \eqref{dist representation}. We omit details for brevity. $\hfill$ $\square$
\medskip

Let us compare regimes (i) and (ii) with the results of \cite{Jiang19}.  
In regime (i), \cite{Jiang19} showed that the statistic in \eqref{regime i} converges to a standard Gumbel distribution when $\rho \sqrt{\log p} \to 0$.  
In our setting, although $\rho \sqrt{\log p} = O\!\left((\log p)^{-1/2}\right) \to 0$, the diverging fourth moment contributes an additional Gaussian component to the limiting distribution.  

In regime (ii), the behavior is more intricate. The data points $\bm{x}_i$ have strongly dependent coordinates with marginals converging to {the Rademacher distribution}.  
Heuristically, the Gaussian contribution in \eqref{dist representation} persists, while the Gumbel component breaks down in the second term of \eqref{dist representation}. Also, the scaling of the statistic in \eqref{regime ii} is of order $\Omega\big( (\log p)^{-1} \big)$, which is nonstandard. This scaling can be regarded as the asymptotic variance. In contrast, it is known (see, e.g., \cite{ding2015multiple}) that for a collection $X_1,\dots,X_N$ of normal distributions with unit variance and arbitrary covariance,
\[
\Var\Big( \max_{1 \leq i \leq N} X_i \Big) \;\geq\; 
\frac{C}{\big( 1 + \mathbb{E}\max_{1 \leq i \leq N} X_i \big)^2},
\]
for some universal constant $C$.  

Thus, in all existing results, the  asymptotic variance of the largest element is of order $(\log p)^{-1}$ since the maximum of $O(p^2)$ standard normal distributions is of size $\Theta(\sqrt{\log p})$. By contrast, in \eqref{regime ii} the asymptotic variance is only of order $\Theta\big( (\log p)^{-2} \big)$, which is caused by the degeneracy effect of $\kappa_n$.

\section{Application III: FWER control in mutiple testing} \label{sec-FWER}
As a direct consequence of Theorem \ref{infty} and  Slepian's lemma, we can see that the test in \eqref{max test} is asymptotically the same as that of the maximum of i.i.d. standard normal variables, provided that \eqref{graphical covariance} is true for some $\delta>0$. Thus, we can pick the threshold $u_n$ as
\[
u_n = \frac{q_{\alpha}}{2\sqrt{\log n}} +2\sqrt{\log n} - \frac{\log \log n  + \log(4\pi)}{4\sqrt{\log p}}
\]
where $q_\alpha$ is the $(1-\alpha)$-quantile of a Gumbel law with CDF $x \mapsto \exp\!\left( -\tfrac{1}{4\sqrt{2\pi}} e^{-x} \right)$.

Moreover, Theorem \ref{infty} also reveals that if $\delta=\delta_n$ is allowed to be $n$-dependent, then the same threshold for $u_n$ would remain valid as long as $\delta_n  \gg \log \log n/\sqrt{\log n}$ as $n\to \infty$. Theorems \ref{critical} and \ref{0} imply that this choice of $u_n$ breaks down when $\delta_n$ decays at least at fast as $\Omega\lb 1/ \log n \rb$. 

We now provide two examples of some popular situations where this type of mutiple testing problem shows up in statistical analysis. The first example is the classical Analysis of Variance (ANOVA), and the second one is a network model known as social relation model (see \cite{hoff2021additive} for more details). 
\smallskip

\textit{Example 1: Two-way ANOVA with random effects.} Analysis of Variance is probably the oldest statistical technique for mean comparison. We shall consider the two-way ANOVA with random effects, which is arguably the most common model used in experimental design. In this context, one considers a collection of observations of random variables $Y_{ij}$'s given by
\[
Y_{ijk}:= \mu_{ij} + \alpha_i + \beta_j + \lb \alpha\beta \rb_{ij} + \epsilon_{ijk}\,, \qquad k=1,\ldots,m.
\]
where $\mu_{ij}$'s are the (deterministic) means. The random variables $\la  \alpha_i; \beta_j; (\alpha\beta)_{ijk}; \epsilon_{ij} \ra$ are i.i.d. normal distributions with possibly different variances. One can think of the pairs $(i,j)$ as treatments and $k$ as the number of replicates of this treatment. In practice, one often takes $\hat{Y}_{ij}$, the sample mean over $k$, to get an estimated treatment effect for every $(i,j)$. It is easy to check that the covariance structure of $\la  \hat{Y}_{ij}; 1\leq i<j\leq n  \ra$
has the same form as in \eqref{sparse equi GF} under normality assumption.

\textit{Example 2: Social relation models.} This is the model taken from \cite{hoff2021additive} (see also the references therein for more details). It is similar to the two-way ANOVA model above, except that $\la i,j \ra$ might have network structures, and can be either directed or undirected. 

\section{Proof of Theorem \ref{infty}} \label{sec-proof 1}
In what follows, we will repeatedly make use of the representation
\begin{align} \label{Gaussian representation}
G_{ij} := \sqrt{r} \lb X_i + X_j \rb + \sqrt{1-2r} \cdot Y_{ij}
\end{align}
for independent sequences $(X_{i})_{1 \leq i \leq n}$ and $(Y_{ij})_{1 \leq i<j \leq n}$ of i.i.d. standard normal random variables.  {It is crucial to note that, using this construction, the field $\{G_{ij}\}_{1 \leq i<j \leq n}$ has the covariance structure \eqref{sparse equi GF}.}

We start by defining $T_n := \log \log n$ and $f_n:= 1-2r$, 
and we introduce the truncation level 
\[
t_n:= \sqrt{2 \log n} + \frac{T_n}{\sqrt{\log n}}.
\]
Writing
$E_n:= \la \max_{1 \leq i \leq n} |X_i| \leq t_n  \ra$,
it is easy to check that 
\begin{align*}
    \mb{P} \lb E_n^c  \rb &\leq 2 n \cdot \lb 1 - \Phi \lb t_n \rb \rb \\
    &\leq 2 n \cdot \frac{e^{-t_n^2/2}}{\sqrt{2\pi} t_n} = O \lb \frac{n \cdot e^{-t_n^2/2}}{\sqrt{\log n}} \rb \\
    &= O \lb \frac{1}{\sqrt{\log n}} \rb \to 0.
\end{align*}
Fix $y \in \mb{R}$ and define 
$$u_n=u_n(y)= \sqrt{2}c_n + \frac{y - \log \lb 4\sqrt{\pi}c_n \rb}{\sqrt{2}c_n}$$
with $c_n$ as in \eqref{dn}. Setting 
\begin{align*}
       A_{ij} = \la G_{ij} \leq u_n  \ra \cup \la |X_i| \vee |X_j| \geq t_n \ra,
\end{align*}
it is easy to see that
\[
\mb{P} \lb \cap_{1\le i<j\le n} A_{ij} \rb \le \mb{P} \lb \max_{1\le i<j\le n} G_{ij} \leq u_n \rb + \mb{P} \lb E_n^c \rb = \mb{P} \lb \max_{1\le i<j\le n} G_{ij} \leq u_n \rb + o(1).
\]
Now, by using Poisson approximation (see, e.g., Theorem 1 in \cite{arratia:goldstein:gordon:1989}), we have
\begin{align*}
    \left| \mb{P} \lb \cap_{1\le i<j\le n} A_{ij} \rb - e^{-\sum_{1\le i<j\le n}p_{ij}} \right| \leq b_1 + b_2\,,
\end{align*}
where 
\begin{align*}
    p_{ij}&:= \mb{P} \lb A_{ij}^c \rb =  \P\Big(\la G_{ij} > u_n \ra \cap \la |X_i| \vee |X_j| < t_n \ra \Big), \\
    b_1&:= O(n^3) \cdot p_{12}^2, \\
    b_2&:= O(n^3) \cdot \mb{P} \lb A_{12}^c \cap A_{13}^c \rb.
\end{align*}
Therefore, we have
\[
 \left|  \mb{P} \lb \max_{1\leq i<j \leq n} G_{ij} \leq u_n \rb- e^{-\sum_{1\le i<j\le n}p_{ij}} \right| \leq b_1 + b_2 + o(1).
\]
To finish the proof, we need to show the following three statements
\begin{align}
    \sum_{1\le i<j\le n} p_{ij} = \frac{n(n-1)}{2} p_{12} &\to e^{-y}\,, \label{mean} \\
    b_1 &\to 0\,, \label{b1} \\
    b_2 &\to 0. \label{b2}
\end{align}

\noindent \textit{\underline{Proof of \eqref{mean}}.} {Using Mill's ratio, we get $\frac{n(n-1)}{2} \,\mb{P} ( G_{12} \geq u_n) \to e^{-y}$; see for example \cite{heiny:mikosch:yslas:2021}. To prove \eqref{mean}, it therefore} suffices to show that 
\[
n^2 \left[ \mb{P} \lb G_{12} \geq u_n \rb - \mb{P} \lb G_{12} \geq u_n, |X_1| \vee |X_2| \leq t_n \rb \right] \to 0.
\]
By noting that
\[
\mb{P} \lb G_{12} \geq u_n \rb - \mb{P} \lb G_{12} \geq u_n, |X_1| \vee |X_2| \leq t_n \rb \leq 2\mb{P} \lb G_{12} \geq u_n, |X_1| \geq t_n \rb,
\]
we only need to prove
\begin{align} \label{M12, |X_1|>=tn}
n^2 \cdot \mb{P} \lb G_{12} \geq u_n, |X_1| \geq t_n \rb \to 0. 
\end{align}
To see this, write
\begin{align*}
 \mb{P} \lb G_{12} \geq u_n, |X_1| \geq t_n \rb  &= \mb{P} \lb G_{12} \geq u_n, X_1 \geq t_n \rb +  \mb{P} \lb G_{12} \geq u_n, X_1 \leq -t_n \rb\,.
\end{align*}
For the second term,  we observe that
\begin{align*}
    \mb{P} \lb G_{12} \geq u_n, X_1 \leq -t_n \rb &= \int_{-\infty}^{-t_n} \left[ 1 - \Phi \lb \frac{u_n-\sqrt{r}x}{\sqrt{1-r}} \rb \right] \phi(x) dx = o(n^{-2}).
\end{align*}
To bound the first term, recall from Example 1 in \cite{Hashorva2003-ms} that 
\begin{align*}
     \mb{P} \lb G_{12} \geq u_n, X_1 \geq t_n \rb & \leq O \lb \frac{1}{\lb u_n -\sqrt{r}t_n \rb \lb t_n - \sqrt{r}u_n \rb} \rb \cdot \exp \lb -\frac{u_n^2-2\sqrt{r} u_nt_n +t_n^2}{2(1-r)} \rb \,.
\end{align*}
Observe that 
\begin{align*}
     \lb u_n - \sqrt{r}t_n \rb \lb t_n - \sqrt{r}u_n \rb 
    & \geq  \Omega\lb \sqrt{\log n} \rb \left[ (1 -\sqrt{2r})\sqrt{\log n} + \frac{T_n}{2\sqrt{\log n}} + o \lb  \frac{1}{\sqrt{\log n}} \rb \right] \\
    & \geq  \Omega \lb (1-\sqrt{2r})\log n \rb = \Omega(f_n)
\end{align*}
where  we have used the fact that $T_n=o\lb \log n \rb$ to get the lower bound $u_n-\sqrt{r}t_n \gtrsim \sqrt{\log n}$. Also, some simple algebra yields
\begin{align*}
    u_n^2&= 4\log n - \log \log n + O \lb \frac{(\log \log n)^2}{\log n} \rb\,, \\
    2\sqrt{r}u_nt_n &= 2\sqrt{r} \cdot 2\sqrt{\log n} \lb 1 - \frac{\log \log n}{8\log n} - O \lb \frac{1}{\log n} \rb \rb  \sqrt{2\log n} \lb 1 + \frac{T_n}{\sqrt{2}\log n} \rb \\
    &=  4\sqrt{2r} \log n \lb 1 + \frac{T_n}{\sqrt{2}\log n} - \frac{\log \log n}{8 \log n} - O \lb \frac{T_n \log \log n}{(\log n)^2} \rb - O \lb \frac{1}{\log n} \rb \rb \,, \\
    t_n^2 &= 2\log n +2\sqrt{2}T_n + \frac{T_n^2}{\log n}\,.
\end{align*}
Thus, we have
\begin{align*}
    -\frac{u_n^2-2\sqrt{r} u_nt_n +t_n^2}{2(1-r)} &\leq -\log n \left[ \frac{6-4\sqrt{2r}}{2(1-r)} \right] +\frac{\log \log n}{2(1-r)} \cdot \lb 1 - \sqrt{\frac{r}{2}} \rb \\
    & \quad- T_n \cdot \frac{2\sqrt{2}-4\sqrt{r}}{2(1-r)} + O(1)
\end{align*}
where we dropped the term involving $T_n^2$ in the above display.  

As a consequence, we obtain
\begin{align*}
    -\frac{u_n^2-2\sqrt{r} u_nt_n +t_n^2}{2(1-r)} &\leq  -\log n \left[ 2 + \frac{6-4\sqrt{2r}-4+2r}{2(1-r)}\right] +O(1) \\
    &= -\log n \left[ 2 +\frac{(1-\sqrt{2r})^2}{2(1-r)} \right] + O(1)
\end{align*}
from which we conclude that
\begin{align*}
    \mb{P} \lb G_{12} \geq u_n, X_1 \geq t_n \rb & \leq O \lb \frac{1}{(1-2r)\log n} \rb  \frac{\exp \left[ -(1+o(1))f_n^2 \log n + O(\log \log n) \right]}{n^2} = o(n^{-2}).
\end{align*}
since $f_n^2\log n \gg \log \log n$. This completes the proof of \eqref{mean}.
\smallskip

\noindent \textit{\underline{Proof of \eqref{b1}}.} This is a direct consequence of \eqref{mean} since the proof of \eqref{M12, |X_1|>=tn} implies that $p_{12}=O(n^{-2})$. 
\smallskip

\noindent \textit{\underline{Proof of \eqref{b2}}.}
Observe that 
\begin{align*}
    \mb{P} \lb A_{12}^c \cap A_{13}^c \rb  &= \mb{P} \lb G_{12} \geq u_n, G_{13} \geq u_n, |X_1| \leq t_n, |X_2|\leq t_n, |X_3| \leq t_n \rb \\
    &\asymp \iiint_{[-t_n,t_n]^3}  \left[ 1- \Phi \lb \frac{u_n-\sqrt{r}(x_1+x_2)}{\sqrt{1-2r}} \rb \right] \cdot \left[ 1- \Phi \lb \frac{u_n-\sqrt{r}(x_1+x_3)}{\sqrt{1-2r}} \rb \right] \\
    & \quad \times \exp \lb -\frac{x_1^2+x_2^2+x_3^2}{2} \rb dx_1dx_2dx_3.
\end{align*}
To bound this integral, let us first do a change of variables to shift the integral to the positive axis.
Setting $\Delta_i:= t_n - x_i$ for  $i=1,2,3$, we get the bound 
\begin{align*}
         \mb{P} \lb A_{12}^c \cap A_{13}^c \rb  
        \asymp & \iiint_{[0,2t_n]^3}   \left[ 1- \Phi \lb L_n + \frac{\sqrt{r}(\Delta_1 + \Delta_2)}{\sqrt{1-2r}} \rb \right] \cdot \left[ 1- \Phi \lb L_n + \frac{\sqrt{r}(\Delta_1 + \Delta_3)}{\sqrt{1-2r}} \rb \right] \\
        & \times \exp \left[ - \frac{\sum_{i=1}^3 \lb  t_n - \Delta_i  \rb^2 }{2} \right] d \bm{\Delta}\,,
\end{align*}
where { $d \bm{\Delta} =d\Delta_1 d\Delta_2 d\Delta_3$ and} 
\begin{align*}
    L_n:= \frac{u_n-2\sqrt{r}t_n}{\sqrt{1-2r}} &= { \frac{2\sqrt{1-2r}\sqrt{\log n}}{1+\sqrt{2r}} }+ O\lb \frac{\log \log n}{\sqrt{f_n \log n}} \rb + O \lb \frac{1}{\sqrt{\log n}} \rb.
\end{align*}

\noindent Under the assumption $f_n \sqrt{\log n} \gg \log \log n$, we have $L_n \geq  \sqrt{f_n \log n}(1+o(1)) \to \infty$. 
By using the elementary bound $1-\Phi(x) \lesssim \exp \lb -x^2/2 \rb$ for all $x\geq 0$, we have 
\begin{align*}
   &\iiint_{[0,2t_n]^3}   \left[ 1- \Phi \lb L_n + \frac{\sqrt{r}(\Delta_1 + \Delta_2)}{\sqrt{1-2r}} \rb \right] \cdot \left[ 1- \Phi \lb L_n + \frac{\sqrt{r}(\Delta_1 + \Delta_3)}{\sqrt{1-2r}} \rb \right] \\
   &\quad \times \exp \left[ - \frac{\sum_{i=1}^3 \lb  t_n - \Delta_i  \rb^2 }{2} \right] d \bm{\Delta} \\
        \lesssim & \exp \lb -L_n^2 - \frac{3t_n^2}{2} \rb\iiint_{[0,2t_n]^3} \exp \left[  -\frac{\sum_{i=1}^3 \Delta_i^2}{2} +t_n \sum_{i=1}^3 \Delta_i - \frac{L_n\sqrt{r}}{\sqrt{1-2r}} \lb 2\Delta_1 + \Delta_2 + \Delta_3  \rb\right] \\
    &\quad \times \exp \left[ - \frac{r}{2(1-2r)} \lb (\Delta_1+ \Delta_2)^2 + \lb \Delta_1 + \Delta_3  \rb^2 \rb\right] d\bm{\Delta}.
\end{align*}
Enlarging the domain of integration to $[0,\infty)^3$, we get the upper bound 
\begin{align*}
       \mb{P} \lb A_{12}^c \cap A_{13}^c \rb \lesssim &  \exp \lb -L_n^2 - \frac{3t_n^2}{2} \rb \times \iiint_{[0,\infty)^3} \exp \left[ -\frac{1}{2}\bm{\Delta}^\top \bm{H} \bm{\Delta} + \bm{g}^{\top}\bm{\Delta} \right] d\bm{\Delta}\,,
\end{align*}
where $\bm{H}:= \bm{I}_3 +\bm{M}$,
\begin{align*}
   \bm{M}:=  \alpha^2  \begin{pmatrix}
        2 & 1 & 1 \\
        1 & 1 & 0 \\
        1 & 0 & 1
    \end{pmatrix}  \qquad \text{ and } \qquad \bm{g} := \begin{pmatrix}
        t_n -2\alpha L_n \\
        t_n -\alpha L_n \\
        t_n - \alpha L_n
    \end{pmatrix}
\end{align*}
with $\alpha:=\sqrt{r/(1-2r)}$. To treat the integral term, note that
\begin{align*}
    & \iiint_{[0,\infty)^3} \exp \left[ -\frac{1}{2}\bm{\Delta}^\top \bm{H} \bm{\Delta} + \bm{g}^{\top}\bm{\Delta} \right] d\bm{\Delta} \\
    = & \iiint_{[0,\infty)^3} \exp \left[ -\frac{1}{2} \lb \bm{\Delta} - \bm{H}^{-1} \bm{g} \rb^\top \bm{H} \lb \bm{\Delta} - \bm{H}^{-1} \bm{g} \rb + \frac{1}{2} \bm{g}^\top \bm{H}^{-1}\bm{g}\right] d\bm{\Delta} \\
     = &  \exp \lb   \frac{1}{2} \bm{g}^\top \bm{H}^{-1}\bm{g} \rb \cdot \iiint_{[0,2t_n]^3-\bm{H}^{-1} \bm{g}}   \exp \lb   \frac{1}{2} \bm{y}^\top \bm{H}\bm{y} \rb d\bm{y} \\
 \leq  &   \exp \lb   \frac{1}{2} \bm{g}^\top \bm{H}^{-1}\bm{g} \rb \cdot \iiint_{[0,\infty)^3-\bm{H}^{-1} \bm{g}}   \exp \lb   \frac{1}{2} \bm{y}^\top \bm{H}\bm{y} \rb d\bm{y}.
\end{align*}
{ where the second equality follows from the change of variable $\bm{y}=\bm{\Delta}-  \bm{H}^{-1} \bm{g}  $.}

One can check that $\bm{H}$ has the three eigenvalues $1+3\alpha^2, 1+\alpha^2,1$, and thus,
\[
\lb \mbox{det}(\bm{H}) \rb^{-1/2} = \frac{1}{\sqrt{(1+3\alpha^2)(1+\alpha^2)}}.
\]
 To find $\bm{H}^{-1}$, it suffices to diagonalize the matrix $\alpha^{-2} \bm{M}$. This matrix has three simple eigenvalues $\lambda_1=3, \lambda_2=1$ and $\lambda_3=0$. The corresponding eigenvectors are 
 \[
 \bm{u}_1:= (2,1,1)^\top; \quad \bm{u}_2:=(0,1,-1)^\top; \quad \bm{u}_3:= (1,-1,-1)^\top.
 \]
Thus, $\bm{g}$ can be written as
\begin{align*}
    \bm{g}= \underbrace{\frac{2t_n-3\alpha L_n}{3}}_{a_1} \bm{u}_1 - \underbrace{\frac{t_n}{3}}_{a_3} \bm{u_3}.
\end{align*}
Consequently,
\begin{align*}
   \bm{H}^{-1} \bm{g} &= \frac{a_1}{1+3\alpha^2} \bm{u}_1  - a_3 \bm{u_3} \\
    &= \frac{2t_n-3\alpha L_n}{3(1+3\alpha^2)} \cdot \lb 2,1,1  \rb^\top  - \frac{t_n}{3} \lb 1,-1,-1 \rb^\top \\
    &= \lb \frac{4t_n-6\alpha L_n}{3(1+3\alpha^2)} - \frac{t_n}{3}, \frac{2t_n-3\alpha L_n}{3(1+3\alpha^2)} + \frac{t_n}{3}, \frac{2t_n-3\alpha L_n}{3(1+3\alpha^2)} + \frac{t_n}{3} \rb^\top \\
    &= \underbrace{\lb \frac{t_n(1-\alpha^2)-2\alpha L_n}{1+3\alpha^2}, \frac{t_n(1+\alpha^2)-\alpha L_n}{1+3\alpha^2}, \frac{t_n(1+\alpha^2)-\alpha L_n}{1+3\alpha^2}  \rb^\top}_{\bm{\mu}:=(\mu_1,\mu_2,\mu_3)^\top}.
\end{align*}
and
\begin{align*}
       \exp \lb   \frac{1}{2} \bm{g}^\top \bm{H}^{-1}\bm{g} \rb  &= \exp \left[ \frac{1}{2} \lb a_1 \bm{u}_1^\top -a_3 \bm{u}_3^\top \rb \bm{H}^{-1}  \lb a_1 \bm{u}_1 -a_3 \bm{u}_3 \rb  \right] \\
       &= \exp \left[ \frac{1}{2} \lb a_1 \bm{u}_1^\top -a_3 \bm{u}_3^\top \rb \bm{H}^{-1}  \lb a_1 \bm{u}_1 -a_3 \bm{u}_3 \rb  \right] \\
       &= \exp \left[  \frac{1}{2} \lb a_1 \bm{u}_1^\top -a_3 \bm{u}_3^\top \rb \lb \frac{a_1 \bm{u}_1}{1+3\alpha^2} - a_3\bm{u}_3 \rb  \right] \\
       &= \exp \lb   \frac{3a_1^2}{1+3\alpha^2} + \frac{3}{2} a_3^2  \rb = \exp \left[ \frac{(2t_n-3\alpha L_n)^2}{3(1+3\alpha^2)} + \frac{t_n^2}{6}\right] .
\end{align*}
The domain of integration after the change of variable is    
$[0,\infty)^3 - \bm{\mu} = \otimes_{i=1}^3 [-\mu_i,\infty)$.
Let $\bm{Z}=(Z_1, Z_2, Z_3)$ be a Gaussian vector in $\mb{R}^3$ with distribution $N \lb  \bm{0}, \bm{H}^{-1} \rb$. One can see that
\begin{align*}
    \iiint_{[0,\infty)^3-\bm{\mu} }   \exp \lb   \frac{1}{2} \bm{y}^\top \bm{H}\bm{y} \rb d\bm{y} &= \mbox{det}(\bm{H}) ^{-1/2} \times \mb{P} \lb  \bm{Z} \geq -\bm{\mu} \rb \\
    &\leq  \frac{\mb{P} \lb  Z_1 \geq -\mu_1 \rb}{\sqrt{(1+3\alpha^2)(1+\alpha^2)}}. 
\end{align*}
By definition of $\mu_1$, we have
 \begin{align*}
     \mu_1 = \frac{-t_n\alpha^2}{1+3\alpha^2} - \frac{2\alpha L_n-t_n}{1+3\alpha^2}.
 \end{align*}
Note that $Z_1 \sim N \lb 0 , \lb \bm{H}^{-1}\rb_{11} \rb$, where
\begin{align*}
    \lb \bm{H}^{-1}\rb_{11} = \lb \lb \bm{I} + \bm{M} \rb^{-1} \rb_{11} = \frac{1}{3} + \frac{2}{3(1+3\alpha^2)} = \frac{1-r}{1+r}.
\end{align*}
Thus, we have established that
\begin{align*}
     \mb{P} \lb A_{12}^c \cap A_{13}^c \rb \lesssim& \frac{ \exp \left[ -L_n^2 -\frac{3t_n^2}{2} + \frac{(2t_n-3\alpha L_n)^2}{3(1+3\alpha^2)} + \frac{t_n^2}{6} \right]}{\sqrt{(1+3\alpha^2)(1+\alpha^2)}}  \mb{P} \lb N \lb 0, \frac{1-r}{1+r} \rb \geq \frac{t_n\alpha^2}{1+3\alpha^2} + \frac{2\alpha L_n-t_n}{1+3\alpha^2} \rb.
\end{align*}
A direct computation shows that
\begin{align*}
    \exp \left[ -L_n^2 -\frac{3t_n^2}{2} + \frac{(2t_n-3\alpha L_n)^2}{3(1+3\alpha^2)} + \frac{t_n^2}{6} \right] &= \exp \lb -\frac{u_n^2}{1+r} \rb, \\
     \frac{t_n\alpha^2}{1+3\alpha^2} + \frac{2\alpha L_n-t_n}{1+3\alpha^2} &= \frac{2\sqrt{r}u_n - (1+r)t_n}{1+r}.
\end{align*}
Thanks to the elementary bound $1-\Phi(x) \lesssim \exp \lb -x^2/2 \rb$ for all $x\geq 0$, we have 
\begin{align*}
    \mb{P} \lb N \lb 0, \frac{1-r}{1+r} \rb \geq \frac{2\sqrt{r}u_n - (1+r)t_n}{1+r} \rb &\lesssim \exp \left[ -\frac{\lb  2\sqrt{r}u_n - (1+r)t_n\rb^2}{2(1+r)^2} \cdot \frac{1+r}{1-r}\right] \\
    &= \exp \left[ -\frac{\lb  2\sqrt{r}u_n - (1+r)t_n\rb^2}{2(1-r^2)} \right].
\end{align*}
Putting everything together, we have
\begin{align}
     n^3 \mb{P} \lb A_{12}^c \cap A_{13}^c \rb \lesssim & \exp \lb 3 \log n -\frac{u_n^2}{1+r} - \frac{\lb  2\sqrt{r}u_n - (1+r)t_n\rb^2}{2(1-r^2)} \rb \nonumber \\
     =&\exp \left[ -g(r) \log n + O \lb \log \log n \rb + O(1)  \right] \label{b2-final}
\end{align}
where
\[
g(r):= \frac{4}{1+r} + \frac{\lb 4\sqrt{r} - \sqrt{2}(1+r)  \rb^2}{2(1-r^2)} -3={ \frac{2(1-\sqrt{2r})^2}{1-r}=\frac{2(1-2r)^2}{(1-r)(1+\sqrt{2r})^2}\ge \frac12f_n^2 }.
\]

The proof is completed by combining \eqref{b2-final}, the display above and the fact that $f_n^2 \log n \gg \log \log n$. $\hfill$ $\square$
\section{Proof of Theorem \ref{critical}} \label{sec-proof 2}
Recall the definitions of $c_n$ and $d_n$ from \eqref{dn} and the representation for $G_{ij}$ in \eqref{Gaussian representation}. Define
\[
Z_{i}:= c_n \lb X_i - d_n \rb.
\]
It is well known that the point process {$N_n := \sum_{i=1}^n \delta_{Z_i} \stackrel{d}{\to} N:= \sum_{i=1}^\infty \delta_{\eta_i}$. The limiting process $N$  a Poisson point process (PPP)} with intensity measure $e^{-x}\,dx$; (see, e.g., \cite{embrechts:kluppelberg:mikosch:1997}).
As a consequence of the point process convergence, for any $k \geq 1$, we have 
\begin{align} \label{convergence point process}
    \lb Z_{(1)}, Z_{(2)}, \dots, Z_{(k)} \rb \stackrel{d}{\to} \lb \eta_1, \eta_2, \dots, \eta_k \rb
\end{align}
where $Z_{(1)}\ge \cdots \ge Z_{(n)}$ denotes the order statistics of $Z_1,\ldots, Z_n$ and $\eta_i$'s are as in \eqref{eta}.

The proof of Theorem \ref{critical} is based on the following result 
\begin{proposition} \label{perturbed point process}
    For any $c>0$ we have 
    \[
    \sup_{1\leq i<j \leq n} \la  \frac{Z_{(i)} + Z_{(j)}}{\sqrt{2}} + cY_{ij} \ra \stackrel{d}{\to}  \sup_{i<j} \la  \frac{\eta_i + \eta_j}{\sqrt{2}} + c\mc{Y}_{ij} \ra 
    \]
    where $\la \eta_i \ra_{i=1}^\infty$ are as in \eqref{eta} and $\mc{Y}_{ij}$'s are i.i.d. standard normals independent from $\eta_i$'s.
\end{proposition}

    It is worth noting that Proposition \ref{perturbed point process} does {\it not} follow directly from a continuous mapping argument. This is because the map 
    \[
    f: \sum_{i} \delta_{x_i} \mapsto \sum_{i<j} \delta_{x_i+x_j}
    \]
    is not almost surely continuous at $N$. In fact, the map $f$ is almost surely discontinuous at $N$. Interestingly, the point process $f(N)$ is locally finite (and thus, is well-defined) but has infinite intensity measure on every compact intervals. To see this, take $[a,b] \subset \mb{R}$ and observe that 
    \begin{align*}
    \mb{E} f(N)[a,b] = \iint_{\mb{R}^2} \mathbf{1}_{\la x+y \in [a,b] \ra} \cdot e^{-x-y} dxdy &= \int_{\mb{R}} e^{-x} \int_{a-x}^{b-x} e^{-y} dy dx \\
    &= \int_{\mb{R}} e^{-x} \lb e^{x-a} - e^{x-b} \rb dx = \infty.
    \end{align*}
    However, the forms of $\eta_i$'s in \eqref{eta} imply that $f(N)$ is locally finite almost surely. An explicit example showing that $f$ is not continuous at any realization of $N$ can be constructed as follows: take $M>0$ large enough and define
    \[
     \mathcal{N}_n := \sum_{i=1}^n \delta_{\eta_i} + \delta_{-\eta_n+ M}.
    \]
    It is easy to see that $\mathcal{N}_n$ converges to $N$ in the vague topology but the image $f(\mathcal{N}_n)$ contains $\delta_M$ for all $n \geq 1$. Therefore, $f(\mathcal{N}_n)$ does not converge to $f(N)$ if $M> \eta_1+\eta_2$. This construction shows that $f$ is not continuous at any point configuration that forms a sequence decreasing to negative infinity.
    \smallskip

\noindent \textbf{Proof of Proposition \ref{perturbed point process}.} For $x\in \mathbb{R}$, define
\begin{align*}
    F_n(x)& := \mb{P} \lb \forall 1\leq i<j \leq n:  \frac{Z_{(i)} + Z_{(j)}}{\sqrt{2}} + cY_{ij} \leq x  \rb, \\
    F(x)& := \mb{P} \lb \forall 1\leq i<j:  \frac{\eta_i + \eta_j}{\sqrt{2}} + c \mc{Y}_{ij} \leq x  \rb.
\end{align*}
It suffices to prove that
\begin{align} \label{limsup}
    \limsup_{n \to \infty} F_n(x) \leq F(x)
\end{align}
and 
\begin{align} \label{liminf}
    \liminf_{n \to \infty} F_n(x) \geq F(x).
\end{align}

\noindent \underline{\textit{Proof of \eqref{limsup}.}} Fix a number $1\le K\le n$ and notice that 
\begin{align*}
    F_n(x) &\leq \mb{P} \lb \forall 1\leq i<j \leq K:  \frac{Z_{(i)} + Z_{(j)}}{\sqrt{2}} + cY_{ij} \leq x  \rb.
\end{align*}
By \eqref{convergence point process} and the continuous mapping theorem, we have 
\[
\lb Z_{(i)} + Z_{(j)} \rb_{1 \leq i<j \leq K} \stackrel{d}{\to} \lb \eta_{i} + \eta_{j} \rb_{1 \leq i<j \leq K}.
\]
Therefore,
\begin{align*}
\limsup_{n \to \infty} F_n(x) &\leq \lim_{n \to \infty} \mb{P} \lb \forall 1\leq i<j \leq K:  \frac{Z_{(i)} + Z_{(j)}}{\sqrt{2}} + cY_{ij} \leq x  \rb \\
&= \mb{P} \lb \forall 1\leq i<j \leq K:  \frac{\eta_{i} + \eta_{j}}{\sqrt{2}} + c\mc{Y}_{ij} \leq x  \rb
\end{align*}
and since this is true for all $K$, we get \eqref{limsup} by letting $K \to \infty$.
\smallskip

\noindent \underline{\textit{Proof of \eqref{liminf}.}} Fix $\ve>0$, by Lemma \ref{localization}, there exists $C_\ve >0$ which depends only on $\ve$ such that
\[
\liminf_{n \to \infty} \mb{P} \lb \argmax_{i<j} \lb  \frac{Z_{(i)} + Z_{(j)}}{\sqrt{2}} + cY_{ij} \rb \in \la 1,2,\dots,K  \ra^2 \rb \geq 1-\ve
\]
for all $K \geq C_\ve$. 
Therefore, for any $K \geq C_\ve$, we have
\begin{align*}
    \liminf_{n \to \infty} F_n(x) &\geq  \liminf_{n \to \infty} \mb{P} \lb \forall 1\leq i<j \leq K:  \frac{Z_{(i)} + Z_{(j)}}{\sqrt{2}} + cY_{ij} \leq x  \rb - \ve \\
    &= \mb{P} \lb \forall 1\leq i<j \leq K:  \frac{\eta_{i} + \eta_{j}}{\sqrt{2}} + c\mc{Y}_{ij} \leq x  \rb - \ve.
\end{align*}
Since this is true for any $K \geq C_\ve$, we get \eqref{liminf} by first letting $K \to \infty$ and then letting $\ve \to 0$. This completes the proof of Proposition \ref{perturbed point process}. $\hfill$ $\square$
\smallskip

We are now ready to prove Theorem~\ref{critical}. Without loss of generality, we may assume that 
\begin{equation}\label{eq:assuse}
\frac{(1-2r)\log n}{2r} = \lambda \in(0,\infty).
\end{equation}
In fact, given a sequence of Gaussian fields $\mc{G}_n$ as in \eqref{sparse equi GF}, with corresponding parameters $r_n$ such that $(1-2r_n)\log n \to \lambda$, one can obtain two-sided bounds on $\mc{G}_n$ by considering the Gaussian fields $\overline{\mc{G}}_n$ and $\underline{\mc{G}}_n$ with parameters $\overline{r}$ and $\underline{r}$, respectively, chosen so that
\begin{align*}
    (1-2\overline{r}) \log n &= \lambda - \ve,  \\
    (1-2\underline{r}) \log n &= \lambda + \ve,
\end{align*}
for some $\ve \in (0,c)$.

This further implies that for all $n$ large enough, $\underline{r} \leq r \leq \overline{r}$. Thus, Slepian's lemma gives
\begin{align*}
    \mb{P} \lb c_n \lb \max_{1\leq i<j \leq n} G_{ij} - \sqrt{2}d_n \rb \leq x \rb &\leq \mb{P} \lb c_n \lb \max_{1\leq i<j \leq n} \overline{G}_{ij} - \sqrt{2}d_n \rb \leq x \rb, \\
     \mb{P} \lb c_n \lb \max_{1\leq i<j \leq n} G_{ij} - \sqrt{2}d_n \rb \leq x \rb &\geq \mb{P} \lb c_n \lb \max_{1\leq i<j \leq n} \underline{G}_{ij} - \sqrt{2}d_n \rb \leq x \rb.
\end{align*}
for all $x \in \mb{R}$. We get the desired conclusion by first letting $n \to \infty$ and then letting $\ve \to 0$.

Now suppose \eqref{eq:assuse}. By Proposition \ref{perturbed point process}, we have
\begin{align*}
    c_n \lb \max_{1\leq i<j \leq n} \frac{G_{ij}}{\sqrt{2r}} - \sqrt{2}d_n  \rb &= \max_{1 \leq i<j \leq n} \la \frac{Z_i+Z_j}{\sqrt{2}} + \sqrt{2\lambda} Y_{ij}  \ra \\
    &\stackrel{d}{=}  \max_{1 \leq i<j \leq n} \la \frac{Z_{(i)}+Z_{(j)}}{\sqrt{2}} + \sqrt{2\lambda} Y_{ij}  \ra \\
    &\stackrel{d}{\to}  \sup_{i<j} \lb \frac{\eta_i+ \eta_j}{\sqrt{2}} + \sqrt{2\lambda} \cdot  Z_{ij} \rb.
\end{align*}
{ Note that the second line in the display above follows from the exchangeability of $\la Y_{ij}; 1\leq i<j\leq n \ra$: if $\pi$ is a permutation of $[n]$ such that $Z_{(i)} = Z_{\pi(i)}$, then $\pi$ is uniformly distributed on the permutation group, independent from $Y_{ij}$'s, and 
\begin{align*}
&\mb{P} \lb \max_{1 \leq i<j \leq n} \la \frac{Z_i+Z_j}{\sqrt{2}} + \sqrt{2\lambda} Y_{ij}  \ra \leq x \rb \\
=& \mb{P} \lb \max_{1 \leq i<j \leq n} \la \frac{Z_{(i)}+Z_{(j)}}{\sqrt{2}} + \sqrt{2\lambda} Y_{\pi(i)\pi(j)}  \ra \leq x \rb \\
=& \mb{E} \left[ \mb{P} \lb  \forall 1\leq i<j \leq n:  \sqrt{2\lambda} \cdot Y_{\pi(i)\pi(j)} \leq  x - \frac{Z_{(i)}+Z_{(j)}}{\sqrt{2}}  \Big| Z_1,\dots,Z_n \rb \right] \\
= & \mb{E} \left[ \mb{P} \lb  \forall 1\leq i<j \leq n:  \sqrt{2\lambda} \cdot Y_{ij} \leq  x - \frac{Z_{(i)}+Z_{(j)}}{\sqrt{2}}  \Big| Z_1,\dots,Z_n \rb \right] \\
= & \mb{P} \lb \max_{1 \leq i<j \leq n} \la \frac{Z_{(i)}+Z_{(j)}}{\sqrt{2}} + \sqrt{2\lambda} Y_{ij}  \ra \leq x \rb.
\end{align*}
}

We now complete the proof of Theorem~\ref{critical} by noting that
\begin{align*}
     c_n \lb \max_{1\leq i<j \leq n} \frac{G_{ij}}{\sqrt{2r}} - \sqrt{2}d_n  \rb & = (1+o(1))\cdot  c_n \lb \max_{1\leq i<j \leq n} G_{ij} - \sqrt{2r}\cdot \sqrt{2}d_n  \rb
\end{align*}
and
\[
c_n \lb 1 - \sqrt{2r}\rb\sqrt{2}d_n = \frac{(1-2r)2 \log(n)\cdot(1+o(1)))}{1+\sqrt{2r}} \to \lambda.
\]
$\hfill$ $\square$

\section{Proof of Theorem \ref{0}} \label{sec-proof 3}
This is a consequence of Theorem \ref{critical} and Slepian's lemma. To see this, consider two Gaussian fields $\overline{\mc{G}}_n$ and $\underline{\mc{G}}_n$ of the form \eqref{sparse equi GF} {with parameters $\overline{r}=1/2$ (resp.) $\underline{r}$ chosen so that $(1-2\underline{r}) \log n =  \ve$}
for some $\ve>0$.

The proof of Theorem \ref{0} now follows from the same Slepian's lemma argument as in the proof of Theorem \ref{critical}, which gives a two-sided bound, then first letting $n \to \infty$ and then letting $\ve \to 0$. $\hfill$ $\square$


\section{Technical results} \label{sec-technical}
Some useful asymptotic expansions will be collected in this Section.

 \begin{lemma} \label{positive definite}
     Define the square matrix $\bm{A}$ of size $p(p-1)/2$ on $I=\la  (i,j): 1\leq i<j \leq p \ra$ as
     \[
     \bm{A}_{\bm{e},\bm{f}} = \begin{cases}
           0, &  \big|\la i,j \ra \cap \la k,l \ra \big| = 0; \\
         b \in [0,1/2), & \bigm| \la i,j \ra \cap \la k,l \ra \big| = 1;\\
          1, & \bigm| \la i,j \ra \cap \la k,l \ra \big| = 2;
     \end{cases}
     \]
     where $\bm{e}=(i,j)$ and $\bm{f}=(k,l)$. Then, for all $p\geq 4$, the smallest eigenvalue of $\bm{A}$ is $\lambda_{min}\lb \bm{A} \rb = 1-2b  $.
 \end{lemma}

\noindent \textbf{Proof of Lemma \ref{positive definite}.} 
We do not have to worry about the particular choice of the ordering on $I$ since $\bm{A}$ is the covariance matrix of   a random vector and {permuting the components of a random vector does not change the smallest eigenvalue of  its covariance matrix.} 

Let $d=p(p-1)/2=|I|$. Define the linearly independent vectors $\bm{s}^{(v)}:=\big(s^{(v)}_{ij}\big) \in \mb{R}^d$, $1\leq v \leq p$ through
 \[
 s^{(v)}_{ij} =  \bm{1}_{\la  v \in \la i,j \ra      \ra}, ~1\le i<j\le p.
 \]

 The linear independence can be seen by observing that for all real numbers $\alpha_i$'s:
 \[
\left[ \sum_{v=1}^p \alpha_v \bm{s}^{(v)} \right]_{(i,j)} = \alpha_i + \alpha_j.
 \]
 From the display above, we also have
 \begin{align} \label{orthogonal}
 \sum_{v=1}^p \bm{s}^{(v)} = 2 \cdot \bm{1}_d.
 \end{align}
Decompose $\mb{R}^d$ into  the direct sum of orthogonal subspaces as
$ \mb{R}^d =  \mbox{span} \la  \bm{1}_d \ra \oplus S_1 \oplus S_2$,
where
\begin{align*}
    S_1& := \la  \sum_{v=1}^p \alpha_v \bm{s}^{(v)}: \sum_{v=1}^p \alpha_v=0 \ra \quad \text{ and } \quad
    S_2 := \Big[  \mbox{span} \la  \bm{1}_d \ra \oplus S_1 \Big]^{\perp}.
\end{align*}
The orthogonality between the first two subspaces is due to \eqref{orthogonal}. The subspace $S_2$ can be characterized equivalently as 
\begin{align*}
S_2= \Big[  \mbox{span} \la  \bm{1}_d \ra \oplus S_1 \Big]^{\perp} &= \mbox{span} \la \bm{s}^{(v)}; 1\leq v \leq p \ra^\perp \\
&= \la \bm{x} \in \mb{R}^d:  \  \sum_{u:(u,v) \ \text{or} \ (v,u) \ \text{in} \ I} x_{(u,v)} = 0 \ \text{for all $v \in [p]$}   \ra.
\end{align*}
We will show below that these are precisely the eigenspaces of $\bm{A}$. On $\mbox{span} \la  \bm{1}_d \ra$, we have
\[
\left[ \bm{A} \bm{1}_d \right]_{(i,j)} = 1 + 2b(p-2)
\]
for all $(i,j) \in I$. This is because for each $(i,j) \in I$, there are $2p-2$ pairs sharing a vertex with it. Therefore, the $\lambda_1=1+2b(p-2)$ is the eigenvalue corresponding to this eigenspace.

On $S_1$, we have for $\bm{x}=\sum_{v=1}^p \alpha_v \bm{s}^{(v)}$ that
\[
 [\bm{A} \bm{x}]_{(i,j)}  = \sum_{v=1}^p \alpha_v [\bm{A} \bm{s}^{(v)}]_{(i,j)} =\begin{cases}
     1 + b(p-2), \ &\text{if} \ v \in \la i,j  \ra; \\
     2b,    &\text{if} \ v \notin \la i,j  \ra.
 \end{cases}
\]
Thus,
\[
\bm{A} \bm{x} = \sum_{v=1}^p \alpha_v \bm{A}\bm{s}^{(v)} = \sum_{v=1}^p \alpha_v \lb 2b\cdot \bm{1}_d + (1+b(p-4)\bm{s}^{(v)}) \rb = \left[ 1 + b(p-4) \right] \cdot \bm{x}.
\]
Consequently, $\lambda_2=1+b(p-4)$ is the eigenvalue on the subspace $S_1$.

Finally, consider $S_2$. For $\bm{x} \in S_2$, we have
\begin{align*}
    [\bm{A}\bm{x}]_{(i,j)} = x_{(i,j)} + \sum_{u \neq i,j} x_{(i,u)} + \sum_{u \neq i,j } x_{(j,u)} = 1 -bx_{(i,j)} - bx_{(i,j)} = (1-2b)x_{(i,j)}.
\end{align*}
Thus $\lambda_3=1-2b$ is the eigenvalue on the subspace $S_2$. 

Among the three eigenvalues $\lambda_1=1+2b(p-2)$, $\lambda_2=1+b(p-4)$, $\lambda_3=1-2b$, the smallest one is $\lambda_3$ if $p\geq 4$. The proof is complete. $\hfill$ $\square$

 \begin{lemma} \label{localization}
    With the same notation as in Proposition~\ref{perturbed point process}, for all $\ve > 0$ there exists $C_\ve > 0$ such that
\begin{align} \label{argmax}
\liminf_{n \to \infty} \mathbb{P} \!\left( \argmax_{(i,j):i<j} \left\{ \frac{Z_{(i)} + Z_{(j)}}{\sqrt{2}} + cY_{ij} \right\} \in [K]\times[K] \right) \geq 1 - \ve
\end{align}
for all $K \geq C_\ve$. Here, $[K] = \la 1,2,\dots,K \ra$.
\end{lemma}
\noindent \textbf{Proof of Lemma \ref{localization}.}
Let $K \geq 1$ and define
\begin{align*}
    M_{in}(K,n)& := \max_{1 \leq i<j \leq K} \la \frac{Z_{(i)} + Z_{(j)}}{\sqrt{2}} + cY_{ij} \ra \quad \text{ and } \quad
    M_{out} (K,n) := \max_{i \vee j \geq K} \la \frac{Z_{(i)} + Z_{(j)}}{\sqrt{2}} + cY_{ij} \ra.
\end{align*}
It suffices to show that 
\[
\limsup_{n \to \infty} \mb{P} \lb M_{out} (K,n) \leq M_{in} (K,n)  \rb \geq 1-\ve
\]
for all $K$ sufficiently large.

Let us first estimate the size of $Z_{(i)}$ for large $i$. Put $b_k:= - (\log k)/2$. Observe that for all $k \leq n$, using Mill's ratio we have
\begin{align}
\mb{E} \left[  N_n \ (b_k,\infty) \right]=n \mb{P} \lb  Z_1 \geq b_k  \rb &=  n\mb{P} \lb  N(0,1) \geq \frac{b_k}{c_n} + d_n \rb \nonumber \\
&\leq \frac{10n}{\sqrt{\log n}} \exp \lb  -\frac{\lb  \frac{b_k}{c_n} + d_n \rb^2}{2} \rb \nonumber \\
&\leq \frac{10n}{\sqrt{\log n}} \exp \left[  b_k \lb \frac{1}{2}+o(1) \rb  -\frac{d_n^2}{2} - \frac{b_n^2}{2c_n^2} \right] \nonumber\\
&\leq  \frac{10n}{\sqrt{\log n}} \exp \left[  b_k \lb \frac{1}{2} +o(1) \rb -\log n + \frac{\log \log n}{2}    \right] \nonumber \\
&\leq k^{3/4} \label{chernoff}
\end{align}
uniformly for all $n$ and $k$ sufficiently large, say $M_0$ (a universal constant). 

Recall that for a binomial distribution $X$ of size $n$ and mean $\mu$, Chernoff's bound yields
\[
\mb{P} \lb X  \geq k \rb \leq \exp \left[ -nD_{KL} \lb \frac{k}{n}, \frac{\mu}{n} \rb  \right]
\]
for all $k \geq \mu $, where the $KL$-divergence is given by
\[
D_{KL}(p,q):= p \log \lb \frac{p}{q}\rb + (1-p) \log \lb \frac{1-p}{1-q} \rb.  
\]
Note that $N_n \lb b_k, \infty \rb$ has a binomial distribution with mean of order $O(k^{3/4})$. Consequently, an application of Chernoff's bound gives
\begin{align*}
\mb{P} \lb  Z_{(k)} \geq b_k  \rb &= \mb{P} \lb  N_n (b_k,\infty)  \geq k  \rb \leq \exp \lb  - \frac{k}{4} \log k + 2k \rb
\end{align*}
and therefore,
    $\mb{P} \lb \exists k \geq K:   Z_{(k)} \geq b_k  \rb \leq \sum_{k\geq K} \exp \lb  - \frac{k}{4} \log k + 2k \rb$.
Thus,  we have
\begin{align} \label{order stat bound}
    \mb{P} \lb \forall k \geq K_1(\ve):   Z_{(k)} \leq b_k  \rb \geq 1-\ve/10,
\end{align}
where 
\begin{align} \label{K_1}
    K_1(\ve) := \min \la K \geq 1: \sum_{k \geq K} \exp \lb  - \frac{k}{4} \log k + 2k \rb \leq \frac{\ve}{10} \ra.
\end{align}

\underline{\textit{Step 1: A high-probability lower bound for $M_{in} (K,n)$.}}
By \eqref{convergence point process}, we have
\begin{align*}
\limsup_{n \to \infty} \mb{P} \lb M_{in}(K,n) \geq x \rb &= \mb{P} \lb \max_{1\leq i<j \leq K} \la  \frac{\eta_{i} + \eta_{j}}{\sqrt{2}} + c\mc{Y}_{ij} \ra \geq x  \rb \geq \mb{P} \lb \frac{\eta_1+\eta_2}{\sqrt{2}} + c\mc{Y}_{12} \geq x \rb
\end{align*}
for all $x\in \mb{R}$ and $K\geq 1$.
Choose $T_\ve \in \mb{R}$ such that 
\begin{align} \label{T}
\mb{P} \lb \frac{\eta_1+\eta_2}{\sqrt{2}} + c\mc{Y}_{12} \geq T_\ve  \rb =1-\frac{\ve}{20}.
\end{align}
With this choice of $T_\ve$, we have 
$\limsup_{n \to \infty} \mb{P} \lb M_{in}(K,n) \geq T_\ve \rb \geq 1 - \frac{\ve}{10}$.
\smallskip

\underline{\textit{Step 2: A high-probability upper bound for $M_{out} (K,n)$.}} Note that
$M_{out} (K,n) = \max \la M_{1n}, M_{2n} \ra$,
where 
\begin{align*}
    M_{1n}& := \max_{K+1\leq i<j \leq n} \la \frac{Z_{(i)} + Z_{(j)}}{\sqrt{2}} + cY_{ij} \ra \quad \text{ and } \quad
    M_{2n} := \max_{K+1\leq j \leq n; 1 \leq i \leq K} \la \frac{Z_{(i)} + Z_{(j)}}{\sqrt{2}} + cY_{ij} \ra.
\end{align*}
We define $W_j := \max_{1\leq i  \leq K} Z_{ij}$
and recall $T_\ve$ in \eqref{T}. By virtue of \eqref{order stat bound}, we have 
\begin{align*}
    \mb{P} \lb  M_{out} (K,n) \geq T_\ve  \rb &\leq \mb{P}\lb M_{1n} \geq T_\ve/2 \rb +  \mb{P}\lb M_{2n} \geq T_\ve/2 \rb \\
    &\leq \frac{\ve}{5} + \mb{P} \lb  \max_{K+1\leq i<j \leq n} \la \frac{-\log(i) - \log(j)}{2\sqrt{2}} + cY_{ij} \ra \geq \frac{T_\ve}{2} \rb \\
    &+  \mb{P} \lb  \frac{Z_{(1)}}{\sqrt{2}} + \max_{j \geq K+1 } \la -\frac{\log(j)}{2\sqrt{2}} + cW_j \ra \geq \frac{T_\ve}{2}  \rb \\
    &\leq \frac{\ve}{5} + \mb{P} \lb  \max_{K+1\leq i<j \leq n} \la \frac{-\log(i) - \log(j)}{2\sqrt{2}} + cY_{ij} \ra \geq \frac{T}{2} \rb \\
    &+ \mb{P} \lb  \frac{Z_{(1)}}{\sqrt{2}}  \geq v \rb +   \mb{P} \lb  \max_{j \geq K+1 } \la -\frac{\log(j)}{2\sqrt{2}} + cW_j \ra \geq \frac{T_\ve}{2} - v  \rb
 \end{align*}
for all $v \in \mb{R}$. 
We proceed by estimating each term on the right-hand side for $K \geq \exp \la -T_\ve(\sqrt{2}+1) \ra$. For the first probability, we have
\begin{align*}
    & \mb{P} \lb  \max_{K+1\leq i<j \leq n} \la \frac{-\log(i) - \log(j)}{2\sqrt{2}} + cY_{ij} \ra \geq \frac{T_\ve}{2} \rb \\ 
    \leq & \sum_{K+1\leq i<j \leq n} \mb{P} \lb  Y_{ij} \geq \frac{1}{2c} \lb T_\ve + \frac{\log(i) + \log(j)}{\sqrt{2}} \rb  \rb \\
    \leq & \sum_{K+1\leq i<j \leq n} \mb{P} \lb  Y_{ij} \geq \frac{\log(i) + \log(j)}{4c} \rb \\
    \leq & \lb  \sum_{i \geq K+1} \exp \lb -\frac{\log^2(i)}{32c^2} \rb   \rb^2 \leq \frac{\ve}{5}\,,
\end{align*}
whenever 
\begin{align} \label{K_2}
    K \geq K_2(\ve):= \min \la K \geq 1: \sum_{i \geq K+1} \exp \lb -\frac{\log^2(i)}{32c^2} \rb  \leq \sqrt{\ve/5} \ra.
\end{align}
For the second probability, note that for $v=v_\ve$ such that $\mb{P} \lb  \eta_1 \geq \sqrt{2}v_\ve \rb= \ve/10$, we have
\begin{align} \label{v_ep}
\limsup_{n \to \infty} \mb{P} \lb  \frac{Z_{(1)}}{\sqrt{2}}  \geq v_\ve \rb = \mb{P} \lb  \eta_1 \geq \sqrt{2}v_\ve \rb= \frac{\ve}{10} < \frac{\ve}{5}
\end{align}
due to \eqref{convergence point process}.
Finally, we estimate the second probability as 
\begin{align*}
      \mb{P} &\lb  \max_{j \geq K+1 } \la -\frac{\log(j)}{2\sqrt{2}} + cW_j \ra \geq \frac{T_\ve}{2} - v_\ve  \rb 
     \leq  \sum_{j \geq K+1} \mb{P} \lb  W_j \geq \frac{1}{c} \lb \frac{T_\ve}{2} - v_\ve + \frac{\log(j)}{2\sqrt{2}} \rb \rb \\
     &\leq  K  \sum_{j \geq K+1} \mb{P} \lb  N(0,1) \geq \frac{1}{c} \lb \frac{T_\ve}{2} - v_\ve + \frac{\log(j)}{2\sqrt{2}} \rb \rb 
     \leq  K \sum_{j \geq K+1} \mb{P} \lb  N(0,1) \geq  \frac{\log(j)}{4c}  \rb
\end{align*}
for all $K \geq \exp \lb 4(\sqrt{2}+1)\lb v_\ve - \frac{T_\ve}{2} \rb \rb$.
The right-hand side above is further bounded by
\[
K \sum_{j \geq K+1} \exp \lb -\frac{\log^2(j)}{32c^2} \rb \leq \frac{\ve}{5}
\]
whenever
\begin{align} \label{K_3}
    K \geq K_3(\ve):= \min \la K \geq 1: K \sum_{j \geq K+1} \exp \lb -\frac{\log^2(j)}{32c^2} \rb \leq \frac{\ve}{5}  \ra.
\end{align}

\underline{\textit{Step 3: Conclusion.}}
With $T_\ve$ chosen in \eqref{T}, we have 
\[
\limsup_{n \to \infty} \mb{P} \lb M_{in}(K,n) \geq T_\ve \rb \geq 1 - \frac{\ve}{10}.
\]
for all $K \geq 1$. 

Moreover, from Step 2, we know that $\mb{P} \lb  M_{out} (K,n) \geq T_\ve  \rb \leq 4\ve/5 $ whenever
\begin{align} \label{C_ep}
K \geq K_{\ve}:=\max \la K_1 \lb \ve \rb, M_0, K_{2}(\ve), K_3(\ve), \exp \la -T_\ve(\sqrt{2}+1) \ra, \exp \la 4(\sqrt{2}+1)\lb v_\ve - \frac{T_\ve}{2} \rb \ra   \ra\,,
\end{align} 
where $M_0$ is the universal constant such that \eqref{chernoff} holds and $K_1(\ve), K_2(\ve), K_3(\ve), T_\ve, v_\ve$ are defined in \eqref{K_1}, \eqref{K_2}, \eqref{K_3}, \eqref{T}, \eqref{v_ep}, respectively. 

We conclude that 
\begin{align*}
    & \liminf_{n \to \infty} \mb{P} \lb  M_{in}(K,n) \geq  M_{out}(K,n) \rb \geq 1 - \frac{\ve}{10} -\frac{4\ve}{5} > 1 -\ve. 
\end{align*}
for all $K \geq K_\ve$, finishing the proof. $\hfill$ $\square$
\smallskip

Note that the proof of Lemma \ref{localization} above actually implies that \eqref{argmax} holds not just at one value of $c$, but also uniformly for all $c$ in a compact interval. The next lemma states that the limiting distribution in Theorem \ref{critical} is non-degenerate.

\begin{lemma} \label{well defined}
  Recall $\eta_i$'s in \eqref{eta}. For $c\geq 0$, define
  \[
  Y_c:= \sup_{i<j} \la \frac{\eta_i+\eta_j}{\sqrt{2}} + cZ_{ij} \ra
  \]
 where $Z_{ij}$'s are standard normal random variables independent of $\eta_i$'s. Then,
  \begin{itemize}
\item[(i)] $Y_c < \infty$ almost surely for all $c \geq 0$.
\item[(ii)]  $Y_{c_n} \stackrel{d}{\to} Y_c$ if $c_n \to c$.
\end{itemize}
\end{lemma}

\noindent \textbf{Proof of Lemma \ref{well defined}.}  Argue similarly to the proof of Lemma \ref{localization} to deduce that for all $\ve,A > 0$, there exists $C_{\ve,A} > 0$ such that
\begin{align} \label{localization 2}
\mathbb{P} \!\left( \argmax_{\la (i,j):i<j \ra} \left\{ \frac{\eta_{i} + \eta_{j}}{\sqrt{2}} + cZ_{ij} \right\} \in [K]\times[K] \right) \geq 1 - \ve
\end{align}
for all $K \geq C_{\ve,A}$ and $c\leq A$. 

We get (i) from the above by letting $\ve \to 0$. To prove (ii), note that 
\[
Y_{c_n}^{(M)}:= \sup_{1\leq i<j \leq M} \la \frac{\eta_i+\eta_j}{\sqrt{2}} + c_nZ_{ij} \ra
\]
 almost surely increases to $Y_{c_n}$ as $M \to \infty$. Thus, for every $x \in \mb{R}$ and $M\geq 1$,
\[
\limsup_{n \to \infty} \mb{P} \lb Y_{c_n} \leq x \rb \leq \limsup_{n \to \infty} \mb{P} \lb Y^{(M)}_{c_n} \leq x \rb = \mb{P} \lb Y^{(M)}_{c} \leq x \rb,
\]
where the last equality follows from the continuous mapping therem. By letting $M \to \infty$, we deduce that
$\limsup_{n \to \infty} \mb{P} \lb Y_{c_n} \leq x \rb \leq \mb{P} \lb Y_{c} \leq x \rb$.
To prove the other direction, use \eqref{localization 2} with $A=\sup_{n\geq 1} c_n$, to obtain that for all $\ve>0$ and $M$ large enough,
\[
\liminf_{n \to \infty} \mb{P} \lb Y_{c_n} \leq x \rb \geq \liminf_{n \to \infty} \mb{P} \lb Y^{M}_{c_n} \leq x \rb -\ve \geq \liminf_{n \to \infty} \mb{P} \lb Y^{M}_{c} \leq x \rb -\ve.
\]
We get (ii) from this by first letting $M \to \infty$ and then letting $\ve \to 0$. $\hfill$ $\square$


\bigskip

\noindent\textbf{Acknowledgments:} 
JH’s research was supported by the Swedish Research Council grant VR-2023-03577 ``High-dimensional extremes and random matrix structures'' and by the Verg-Foundation.

\bibliographystyle{plainnat}
\bibliography{paper-ref}

\end{document}